\documentclass[11pt,reqno,twoside]{amsart}

\usepackage[T1]{fontenc}
\usepackage[utf8]{inputenc}
\usepackage{lmodern}
\usepackage{microtype}

\usepackage[
a4paper,
inner=0.8in,
outer=0.8in,
top=1.2in,
bottom=1.4in
]{geometry}

\usepackage{amsmath,amssymb,amsthm,mathtools}
\usepackage{bm}
\usepackage{mathrsfs}
\usepackage{dsfont}
\usepackage{cite}

\usepackage{graphicx,tikz,caption,subcaption}
\usepackage{enumitem}

\usepackage[
colorlinks=true,
linkcolor=blue!50!black,
citecolor=green!40!black,
urlcolor=magenta!60!black
]{hyperref}
\usepackage[nameinlink,capitalize]{cleveref}

\numberwithin{equation}{section}

\theoremstyle{plain}
\newtheorem{theorem}{Theorem}[section]

\theoremstyle{plain}
\newtheorem*{theoremA}{Theorem A}
\newtheorem*{theoremB}{Theorem B}
\newtheorem*{theoremC}{Theorem C}

\theoremstyle{definition}
\newtheorem{lemma}[theorem]{Lemma}
\newtheorem{proposition}[theorem]{Proposition}
\newtheorem{corollary}[theorem]{Corollary}

\theoremstyle{remark}
\newtheorem{remark}[theorem]{Remark}

\title[The sharp constant in the second-order Hydrogen Uncertainty Principle]
{The Sharp Constant in the Second-Order Hydrogen Uncertainty Principle:
Spectral Characterization and Quantitative Approximation}

\author[M. Jleli]{Mohamed Jleli}
\address{(M. Jleli) Department of Mathematics, College of Science, King Saud University, Riyadh 11451, Saudi Arabia}
\email{jleli@ksu.edu.sa}

\author[B. Samet]{Bessem Samet}
\address{(B. Samet) Department of Mathematics, College of Science, King Saud University, Riyadh 11451, Saudi Arabia}
\email{bessem.samet@gmail.com}

\date{}

\begin{document}

\begin{abstract}
In this work, we study the sharp constant \(\mathcal H_d\) in the
second-order Hydrogen Uncertainty Principle in every dimension \(d\ge2\).
Using a spectral basis built from generalized Laguerre polynomials,
together with a linearization procedure and spherical harmonic
decomposition, we establish a spectral characterization of
\(\mathcal H_d\). More precisely, we show that
\[
\mathcal H_d
=
\frac12\min\{d+1,c_d^*\},
\]
where \(c_d^*\in(0,d+3)\) is the unique parameter such that
\[
\|S_{c_d^*}^{(d)}\|_{\ell^2\to\ell^2}=1.
\]
Here, for \(0<c<d+3\), \(S_c^{(d)}\) is an explicit positive
Hilbert--Schmidt operator on \(\ell^2(\mathbb N_0)\).
This characterization reveals a competition between the radial sector,
which yields the value \(d+1\), and the first non-radial sector, encoded
by the spectral parameter \(c_d^*\). For \(d\ge4\), we recover the known
identity
\[
\mathcal H_d=\frac{d+1}{2}
\]
through a spectral argument, giving an alternative proof of the results
of Cazacu, Flynn and Lam \cite{CazacuFlynnLam} for \(d\ge5\) and Huang
and Ye \cite{HuangYe} for \(d=4\). For \(d\in\{2,3\}\), this characterization yields
\[
\mathcal H_d=\frac{c_d^*}{2}.
\]
We then exploit this spectral formulation to derive an explicit finite
recursive scheme yielding effective two-sided bounds for
\(\mathcal H_2\) and \(\mathcal H_3\), substantially improving the
estimates of Huang and Ye \cite{HuangYe}.
\end{abstract}



\maketitle

\section{Introduction and main results}

Uncertainty principles form a fundamental theme in harmonic analysis and
mathematical physics; see, for instance, \cite{Fefferman,FollandSitaram}
and the references therein. Originating in quantum mechanics, they provide
quantitative lower bounds relating different measures of concentration of
a function. Beyond their physical interpretation, uncertainty principles
are closely connected with several classical families of functional
inequalities and naturally lead to questions concerning sharp constants
and extremal functions.

The Hydrogen Uncertainty Principle is distinguished by the singular weight
\(|x|^{-1}\), which reflects the spatial dependence of the Coulomb
potential in hydrogenic systems; see
\cite{FrohlichLiebLoss,CazacuFlynnLam}. This connection accounts for
both its terminology and its relevance in mathematical physics,
particularly in questions related to the energy and stability of
hydrogenic atoms; see \cite{FrohlichLiebLoss,LiebSeiringer}.

The first-order Hydrogen Uncertainty Principle states that, for \(d\ge2\)
and \(u\in H^1(\mathbb R^d)\),
\[
\left(
\int_{\mathbb R^d}|\nabla u|^2\,dx
\right)
\left(
\int_{\mathbb R^d}|u|^2\,dx
\right)
\ge
\frac{(d-1)^2}{4}
\left(
\int_{\mathbb R^d}\frac{|u|^2}{|x|}\,dx
\right)^2.
\]
The constant \((d-1)^2/4\) is sharp, and equality is attained by
functions of the form
\[
u(x)=a e^{-b|x|},
\qquad
a\in\mathbb R\setminus\{0\},\quad b>0;
\]
see, for instance, \cite{CazacuFlynnLam,Frank2011}.
This inequality can be viewed as a particular case of the
Caffarelli--Kohn--Nirenberg family of weighted interpolation inequalities
introduced in \cite{CaffarelliKohnNirenberg}. We refer to
\cite{CatrinaWang,CatrinaCosta,Costa2008,CazacuFlynnLam2021,
CazacuFlynnLamLu2024} and the references therein for results concerning
sharp constants, extremal functions, refinements, and stability.

We now turn to the second-order Hydrogen Uncertainty Principle, which is
the main subject of this paper. For \(d\ge2\), let \(\mathcal H_d\)
denote the optimal constant in
\begin{equation}\label{eq:HyUP}
\left(
\int_{\mathbb R^d}|\Delta u|^2\,dx
\right)
\left(
\int_{\mathbb R^d}|\nabla u|^2\,dx
\right)
\ge
\mathcal H_d^2
\left(
\int_{\mathbb R^d}
\frac{|\nabla u|^2}{|x|}\,dx
\right)^2,
\qquad
u\in W^{2,2}(\mathbb R^d).
\end{equation}
Equivalently,
\begin{equation}\label{variational-HN}
\mathcal H_d
=
\inf_{u\in W^{2,2}(\mathbb R^d)\setminus\{0\}}
\frac{
\left(
\displaystyle\int_{\mathbb R^d}|\Delta u|^2\,dx
\right)^{1/2}
\left(
\displaystyle\int_{\mathbb R^d}|\nabla u|^2\,dx
\right)^{1/2}
}{
\displaystyle\int_{\mathbb R^d}
\frac{|\nabla u|^2}{|x|}\,dx
}.
\end{equation}

The following result was proved by Cazacu, Flynn and Lam
\cite{CazacuFlynnLam}.

\begin{theoremA}
Let \(d\ge5\). Then
\begin{equation}\label{HN-N5}
\mathcal H_d=\frac{d+1}{2}.
\end{equation}
Moreover, equality in \eqref{eq:HyUP} is attained by functions of the form
\[
u(x)=a(1+b|x|)e^{-b|x|},
\qquad
a\in\mathbb R\setminus\{0\},\quad b>0.
\]
\end{theoremA}

Their argument relies on spherical harmonic decomposition and on
sectorwise estimates obtained through weighted Hardy inequalities. For
\(d\ge5\), the resulting Hardy-based lower bounds are minimized by the
radial sector \(k=0\), yielding \eqref{HN-N5}. For \(2\le d\le4\),
however, the corresponding lower bound in the sector \(k=1\) is strictly
smaller than the radial value, and this argument therefore does not yield
the expected sharp constant. In the same paper, Cazacu, Flynn and Lam
\cite{CazacuFlynnLam} conjectured that Theorem~A remains valid in the
lower dimensions \(2\le d\le4\).

Recently, Chen and Tang \cite{ChenTang} disproved this conjecture
in dimensions \(d=2\) and \(d=3\). Their result can be stated as follows.

\begin{theoremB}
The following bounds hold:
\[
\frac12
\le
\mathcal H_2
\le
\frac{\sqrt3}{2},
\]
and
\[
\frac32
\le
\mathcal H_3
\le
\frac{2\sqrt{21}}{5}.
\]
In particular,
\[
\mathcal H_d<\frac{d+1}{2},
\qquad
d\in\{2,3\}.
\]
\end{theoremB}

The lower bounds are obtained from the sectorwise estimates of Cazacu,
Flynn and Lam by minimizing the resulting bounds over the spherical
harmonic degree, whereas the upper bounds are obtained by evaluating the
variational quotient at a suitable test function in the first non-radial
sector \(k=1\).

More recently, Huang and Ye \cite{HuangYe} gave a unified treatment of
the second-order Hydrogen Uncertainty Principle in dimensions \(d\ge4\),
recovering Theorem~A for \(d\ge5\) and proving that it remains valid in
the previously open case \(d=4\). They also improved the lower bounds of
Chen and Tang \cite{ChenTang} in dimensions \(d=2\) and \(d=3\), and
established the following result.

\begin{theoremC}
The following bounds hold:
\begin{equation}\label{Estimate-Huang-Ye-H2}
\frac{3+6\sqrt2}{14}
\le
\mathcal H_2
\le
\frac{\sqrt3}{2},
\end{equation}
and
\begin{equation}\label{Estimate-Huang-Ye-H3}
\frac74
\le
\mathcal H_3
\le
\frac{2\sqrt{21}}{5}.
\end{equation}
Moreover, the optimal constants \(\mathcal H_2\) and \(\mathcal H_3\)
are attained in the closures of \(C_c^\infty(\mathbb R^2)\) and
\(C_c^\infty(\mathbb R^3)\), respectively, with respect to the norm
\[
u\longmapsto
\|\Delta u\|_2+\|\nabla u\|_2.
\]
\end{theoremC}

Their approach is based on an equivalent linearized formulation combined
with spherical harmonic decomposition. In dimensions \(d=2\) and \(d=3\),
the determination of the sharp constant is reduced to a one-dimensional
variational problem associated with the first non-radial sector \(k=1\).
The improved lower bounds are then obtained through suitable
transformations of the radial profile together with weighted Hardy-type
inequalities and auxiliary one-dimensional estimates. 

In this paper, we establish a spectral characterization of the sharp
constant \(\mathcal H_d\) in every dimension \(d\ge2\). For \(d\ge4\),
this characterization recovers the known identity
\[
\mathcal H_d=\frac{d+1}{2}
\]
through an alternative, purely spectral approach. In dimensions \(d=2\)
and \(d=3\), it leads to substantially sharper two-sided bounds for
\(\mathcal H_2\) and \(\mathcal H_3\) than those previously available.

Before stating our main results, we introduce some notation that will be
used throughout the paper. We set
\[
\mathbb N=\{1,2,3,\ldots\},
\qquad
\mathbb N_0=\mathbb N\cup\{0\},
\]
and denote by \(\|\cdot\|_2\) the \(L^2\)-norm over the underlying
domain whenever the latter is clear from the context.

If \(E\) and \(F\) are normed spaces and \(T:E\to F\) is a bounded
linear operator, we write
\[
\|T\|_{E\to F}
=
\sup_{x\in E\setminus\{0\}}
\frac{\|Tx\|_F}{\|x\|_E}.
\]

We denote by
\[
\ell^2(\mathbb N_0)
=
\left\{
a=(a_j)_{j\ge0}:
a_j\in\mathbb R,\quad
\sum_{j=0}^\infty |a_j|^2<\infty
\right\}
\]
the usual Hilbert space, with inner product
\[
\langle a,b\rangle_{\ell^2}
=
\sum_{j=0}^\infty a_jb_j
\]
and norm
\[
\|a\|_{\ell^2}
=
\left(
\sum_{j=0}^\infty |a_j|^2
\right)^{1/2}.
\]
We denote by \((e_j)_{j\ge0}\) its canonical orthonormal basis.

For a bounded linear operator
\[
T:\ell^2(\mathbb N_0)\to\ell^2(\mathbb N_0),
\]
we write
\[
T=(t_{jk})_{j,k\ge0},
\qquad
t_{jk}
=
\langle Te_k,e_j\rangle_{\ell^2},
\qquad
j,k\in\mathbb N_0,
\]
for its matrix representation with respect to \((e_j)_{j\ge0}\).
If \(T\) is Hilbert--Schmidt, then
\[
\|T\|_{\mathrm{HS}}^2
=
\sum_{j,k=0}^\infty |t_{jk}|^2.
\]
Finally, \(T\) is said to be positive if it is self-adjoint and
\[
\langle Ta,a\rangle_{\ell^2}\ge0,
\qquad
a\in\ell^2(\mathbb N_0).
\]

For \(d\ge2\) and \(j\in\mathbb N_0\), set
\begin{equation}\label{eq-hj-gammaj}
h_j^{(d)}
=
\frac{(d-1)^2}{2^{d+1}}
\sum_{m=0}^j
\frac{
\Gamma(2m+d+1)
}{
\Gamma(2m+1)(4m+d-1)(4m+d+3)
},
\end{equation}
where \(\Gamma\) denotes the Gamma function.

For \(0<c<d+3\), define
\begin{equation}\label{eq-gammajd}
\gamma_j^{(d)}(c)
=
\frac{
2^{d+3}c(4j+d+3)^2\Gamma(2j+1)
}{
(d-1)^2(4j+d+3-c)\Gamma(2j+d+3)
},
\qquad
j\in\mathbb N_0.
\end{equation}

We then introduce the matrix
\begin{equation}\label{eq-Sc}
S_c^{(d)}
=
\left(s_{jk}^{(d)}(c)\right)_{j,k\ge0},
\end{equation}
where
\begin{equation}\label{eq-sjk}
s_{jk}^{(d)}(c)
=
\sqrt{\gamma_j^{(d)}(c)}\,
h_{\min\{j,k\}}^{(d)}\,
\sqrt{\gamma_k^{(d)}(c)},
\qquad
j,k\in\mathbb N_0.
\end{equation}
As shown in Lemma~\ref{L5.8}, this matrix defines a positive
Hilbert--Schmidt operator
\[
S_c^{(d)}:\ell^2(\mathbb N_0)\longrightarrow\ell^2(\mathbb N_0),
\]
and hence a bounded compact operator. In particular,
\[
\left(S_c^{(d)}a\right)_j
=
\sum_{k=0}^\infty
s_{jk}^{(d)}(c)a_k,
\qquad
j\in\mathbb N_0,
\]
for every \(a=(a_k)_{k\ge0}\in\ell^2(\mathbb N_0)\).

We are now in a position to state our first main result.

\begin{theorem}\label{thm:spectral-characterization}
Let \(d\ge2\). There exists a unique \(c_d^*\in(0,d+3)\) such that
\[
\|S_{c_d^*}^{(d)}\|_{\ell^2\to\ell^2}=1.
\]
Moreover,
\begin{equation}\label{Hd-characterization}
\mathcal H_d
=
\frac12\min\{d+1,c_d^*\}.
\end{equation}
\end{theorem}

Theorem~\ref{thm:spectral-characterization} provides a unified spectral
characterization of the sharp constant \(\mathcal H_d\) in every
dimension \(d\ge2\), reducing its determination to the spectral analysis
of the explicit positive compact operator \(S_c^{(d)}\). In particular,
\eqref{Hd-characterization} reveals a competition between the radial
value \(d+1\) and the spectral quantity \(c_d^*\) associated with the
first non-radial sector.

Our approach starts from a linearized formulation of
\eqref{variational-HN}, followed by spherical harmonic decomposition.
The higher spherical harmonic sectors are shown not to contribute to the
minimum, reducing the problem to a competition between the radial sector
and the first non-radial sector. The latter is then reduced to a
one-dimensional variational problem. The main new ingredient is a further
spectral reduction based on a family constructed from generalized
Laguerre polynomials. This basis diagonalizes the local quadratic forms
arising in the one-dimensional problem. After separating the even and
odd components of the resulting expansion, the odd component is shown
not to affect the optimal constant, whereas the even component leads to
the positive compact operator \(S_c^{(d)}\). The resulting discrete
spectral formulation yields \eqref{Hd-characterization}.

As a first consequence of Theorem~\ref{thm:spectral-characterization},
we recover the sharp constant in dimensions \(d\ge4\).

\begin{corollary}\label{cor:Hd-dge4}
Let \(d\ge4\). Then
\[
c_d^*>d+1,
\]
and consequently
\[
\mathcal H_d=\frac{d+1}{2}.
\]
\end{corollary}

In contrast, for \(d=2\) and \(d=3\), the spectral parameter \(c_d^*\)
lies strictly below the radial value \(d+1\), and therefore determines
the sharp constant.

\begin{corollary}\label{cor:Hd-d23}
Let \(d\in\{2,3\}\). Then
\[
c_d^*<d+1,
\]
and consequently
\begin{equation}\label{Hd-d-2-3}
\mathcal H_d=\frac{c_d^*}{2}.
\end{equation}
\end{corollary}

Corollary~\ref{cor:Hd-d23} reduces the determination of the sharp
constants \(\mathcal H_2\) and \(\mathcal H_3\) to that of the spectral
parameters \(c_2^*\) and \(c_3^*\). We next develop an explicit
finite-step scheme for approximating these parameters.

For \(d\ge2\), set
\begin{equation}\label{eq-hminus1-d}
h_{-1}^{(d)}=0,
\end{equation}
and, for \(j\in\mathbb N_0\), define
\begin{equation}\label{eq-epsilon-d}
\varepsilon_j^{(d)}
=
h_j^{(d)}-h_{j-1}^{(d)}.
\end{equation}
By \eqref{eq-hj-gammaj},
\begin{equation}\label{exp-epsilon-md}
\varepsilon_j^{(d)}
=
\frac{(d-1)^2}{2^{d+1}}
\frac{
\Gamma(2j+d+1)
}{
\Gamma(2j+1)(4j+d-1)(4j+d+3)
}
>0,
\qquad
j\in\mathbb N_0.
\end{equation}
Using the recurrence relation
\begin{equation}\label{eq-gamma-functional-relation}
\Gamma(z+1)=z\Gamma(z),
\qquad z>0,
\end{equation}
we obtain, in particular,
\begin{equation}\label{eq-epsilon-d-explicit}
\varepsilon_j^{(2)}
=
\frac{(j+1)(2j+1)}
{4(4j+1)(4j+5)},
\qquad
\varepsilon_j^{(3)}
=
\frac{j+1}{8},
\qquad
j\in\mathbb N_0.
\end{equation}

For \(d\in\{2,3\}\) and \(0<c<d+1\), define the sequences
\[
\bigl(\mathbb A_j^{(d)}(c)\bigr)_{j\in\mathbb N_0},
\qquad
\bigl(\mathbb F_j^{(d)}(c)\bigr)_{j\in\mathbb N_0},
\]
by
\begin{equation}\label{eq-shooting-initial-d}
\mathbb A_0^{(d)}(c)=1,
\qquad
\mathbb F_0^{(d)}(c)=\frac{1}{h_0^{(d)}},
\end{equation}
and, for \(j\in\mathbb N_0\),
\begin{align}
\mathbb F_{j+1}^{(d)}(c)
&=
\mathbb F_j^{(d)}(c)
-
\gamma_j^{(d)}(c)\mathbb A_j^{(d)}(c),
\label{eq-shooting-F-d}\\
\mathbb A_{j+1}^{(d)}(c)
&=
\mathbb A_j^{(d)}(c)
+
\varepsilon_{j+1}^{(d)}
\mathbb F_{j+1}^{(d)}(c).
\label{eq-shooting-A-d}
\end{align}
In particular,
\[
\mathbb F_0^{(2)}(c)=20,
\qquad
\mathbb F_0^{(3)}(c)=8.
\]

For \(m\in\mathbb N_0\), set
\begin{equation}\label{eq-frakb-d}
\mathfrak b_m^{(d)}(c)
=
\begin{cases}
\displaystyle
\frac{
8c(4m+5)(4m+9)
}{
(m+2)(2m+3)
\left(
16m^2+56m+45-(4m+7)c
\right)
},
& d=2,
\\[4mm]
\displaystyle
\frac{
16c
}{
(m+2)
\left(
24m^2+96m+90-(6m+11)c
\right)
},
& d=3,
\end{cases}
\qquad
0<c<d+1.
\end{equation}
The denominators in \eqref{eq-frakb-d} are strictly positive for
\(m\in\mathbb N_0\) and \(0<c<d+1\).

For \(d\in\{2,3\}\) and \(m\in\mathbb N_0\), we then consider the scalar
equation
\begin{equation}\label{eq-cdm-def}
\mathbb F_{m+1}^{(d)}(c)
=
\mathfrak b_m^{(d)}(c)
\mathbb A_m^{(d)}(c),
\qquad
0<c<d+1.
\end{equation}

We now state our second main result.

\begin{theorem}\label{thm:quantitative-estimate}
The following statements hold:
\begin{itemize}
	\item[\rm(i)] For each \(d\in\{2,3\}\) and \(m\in\mathbb N_0\),
	\eqref{eq-cdm-def} admits a smallest solution in \((0,d+1)\),
	denoted by \(\underline c_{d,m}\).

	\item[\rm(ii)] For each \(d\in\{2,3\}\),
	\begin{equation}\label{eq-cdm-convergence}
	\underline c_{d,m}\uparrow c_d^*
	\qquad\text{as }m\to\infty.
	\end{equation}

	\item[\rm(iii)] For every \(m\in\mathbb N\),
	\begin{align}
	0<c_2^*-\underline c_{2,m}
	&<
	\frac{2}{15m^3},
	\label{eq-c2-effective-error}\\
	0<c_3^*-\underline c_{3,m}
	&<
	\frac{27}{100m^4}.
	\label{eq-c3-effective-error}
	\end{align}
\end{itemize}
\end{theorem}

Theorem~\ref{thm:quantitative-estimate} provides a constructive
approximation of the critical spectral parameters \(c_2^*\) and \(c_3^*\).
For each fixed \(m\), the quantities
\(\mathbb A_m^{(d)}(c)\) and \(\mathbb F_{m+1}^{(d)}(c)\) are obtained
after finitely many steps of the recursion
\eqref{eq-shooting-initial-d}--\eqref{eq-shooting-A-d}, and
\(\underline c_{d,m}\) is then determined as the smallest solution of
the scalar equation \eqref{eq-cdm-def} in the known interval
\((0,d+1)\). Thus, its computation does not require prior knowledge of
\(c_d^*\). The resulting sequence provides strictly increasing rigorous
lower bounds converging to \(c_d^*\), while
\eqref{eq-c2-effective-error}--\eqref{eq-c3-effective-error} give
explicit error estimates of order \(m^{-3}\) in dimension \(2\) and
\(m^{-4}\) in dimension \(3\).

As an immediate consequence of Corollary~\ref{cor:Hd-d23} and
Theorem~\ref{thm:quantitative-estimate}, we obtain effective two-sided
bounds for the sharp constants.

\begin{corollary}\label{cor:effective-Hd}
For every \(m\in\mathbb N\),
\begin{align}
\frac{\underline c_{2,m}}{2}
<
\mathcal H_2
&<
\frac{\underline c_{2,m}}{2}
+
\frac{1}{15m^3},
\label{eq-effective-H2}\\
\frac{\underline c_{3,m}}{2}
<
\mathcal H_3
&<
\frac{\underline c_{3,m}}{2}
+
\frac{27}{200m^4}.
\label{eq-effective-H3}
\end{align}
\end{corollary}

Thus, \eqref{eq-effective-H2}--\eqref{eq-effective-H3} provide
computable two-sided approximations of the sharp constants
\(\mathcal H_2\) and \(\mathcal H_3\). The lower bounds are obtained
after finitely many explicit recursive steps, while the corresponding
upper bounds follow from the explicit error estimates. Consequently,
the scheme yields approximations of \(\mathcal H_2\) and
\(\mathcal H_3\) with arbitrarily prescribed accuracy.

To illustrate the effectiveness of the recursive scheme, we apply
Corollary~\ref{cor:effective-Hd} with \(m=2\). In this case, the scalar
equations \eqref{eq-cdm-def} reduce to explicit algebraic equations,
leading to the following numerical bounds for the sharp constants.

\begin{corollary}\label{cor:numerical-Hd}
The sharp constants \(\mathcal H_2\) and \(\mathcal H_3\) satisfy
\begin{equation}\label{H2-est-m2}
0.8445023006
<
\mathcal H_2
<
0.8528356340
\end{equation}
and
\begin{equation}\label{H3-est-m2}
1.8233104714
<
\mathcal H_3
<
1.8317479715.
\end{equation}
\end{corollary}

Compared with Theorem~C, already for \(m=2\),
Corollary~\ref{cor:numerical-Hd} yields substantially sharper two-sided
bounds than \eqref{Estimate-Huang-Ye-H2} and
\eqref{Estimate-Huang-Ye-H3}. Indeed, for \(d=2\),
\[
\frac{3+6\sqrt2}{14}
<
0.8445023006
<
\mathcal H_2
<
0.8528356340
<
\frac{\sqrt3}{2}.
\]
Moreover, the width of the new enclosing interval is less than one fifth
of that in \eqref{Estimate-Huang-Ye-H2}.

Similarly, for \(d=3\),
\[
\frac74
<
1.8233104714
<
\mathcal H_3
<
1.8317479715
<
\frac{2\sqrt{21}}{5},
\]
and the width of the new enclosing interval is less than one ninth of
that in \eqref{Estimate-Huang-Ye-H3}.

Sharper bounds can of course be obtained by taking larger values of
\(m\) in Corollary~\ref{cor:effective-Hd}. The choice \(m=2\) in
Corollary~\ref{cor:numerical-Hd} is intended only to provide a concrete
quantitative illustration of the improvement achieved by the recursive
scheme over the previously known estimates.

The rest of the paper is organized as follows. In Section~\ref{sec2}, we
introduce the functional framework associated with \eqref{eq:HyUP} and
derive an equivalent linearized variational formulation. In
Section~\ref{sec3}, we develop the spherical harmonic decomposition and
reduce the problem to the relevant angular sectors. In
Section~\ref{sec4}, we carry out the one-dimensional reduction of the
first non-radial sector. In Section~\ref{sec5}, we introduce the
Laguerre-based spectral basis, derive the discrete operator
\(S_c^{(d)}\), and prove Theorem~\ref{thm:spectral-characterization}
together with Corollaries~\ref{cor:Hd-dge4} and~\ref{cor:Hd-d23}.
Finally, in Section~\ref{sec6}, we develop the finite recursive scheme
for the critical spectral parameters \(c_2^*\) and \(c_3^*\), prove
Theorem~\ref{thm:quantitative-estimate}, and establish
Corollaries~\ref{cor:effective-Hd} and~\ref{cor:numerical-Hd}.

\section{Functional framework and linearization}\label{sec2}

Throughout this section, let \(d\ge2\). We begin by introducing the
natural energy space associated with the quantities appearing in
\eqref{eq:HyUP}. For \(u\in C_c^\infty(\mathbb R^d)\), set
\[
A(u)=\|\Delta u\|_2^2,
\qquad
B(u)=\|\nabla u\|_2^2,
\qquad
J(u)=\int_{\mathbb R^d}
\frac{|\nabla u(x)|^2}{|x|}\,dx,
\]
and
\[
\mathcal N_d(u)
=
\left(
A(u)+B(u)
\right)^{1/2}.
\]
We define \(\mathcal X_d\) as the completion of
\(C_c^\infty(\mathbb R^d)\) with respect to the norm \(\mathcal N_d\),
and use the same notation \(A\) and \(B\) for their natural extensions
to \(\mathcal X_d\).

Since \(\mathcal H_d>0\), \eqref{eq:HyUP} gives
\[
J(u)
\le
\frac{1}{\mathcal H_d}\sqrt{A(u)B(u)}
\le
\frac{1}{2\mathcal H_d}\mathcal N_d(u)^2,
\qquad
u\in C_c^\infty(\mathbb R^d).
\]
Consequently,
\[
\left\|
\frac{\nabla u}{|x|^{1/2}}
\right\|_2
\le
\frac{1}{\sqrt{2\mathcal H_d}}\,
\mathcal N_d(u),
\qquad
u\in C_c^\infty(\mathbb R^d).
\]
Thus the map
\[
u\longmapsto \frac{\nabla u}{|x|^{1/2}}
\]
extends continuously from \(\mathcal X_d\) into
\(L^2(\mathbb R^d;\mathbb R^d)\). Accordingly, \(J\) extends to
\(\mathcal X_d\) by
\[
J(u)
=
\left\|
\frac{\nabla u}{|x|^{1/2}}
\right\|_2^2,
\qquad
u\in\mathcal X_d,
\]
and we keep the same notation for this extension.

\begin{lemma}\label{L2.1}
The sharp constant \(\mathcal H_d\) admits the variational representation
\begin{equation}\label{Hd-Xd-representation}
\mathcal H_d
=
\inf_{u\in\mathcal X_d\setminus\{0\}}
\frac{\sqrt{A(u)B(u)}}{J(u)}.
\end{equation}
\end{lemma}

\begin{proof}
Since \(C_c^\infty(\mathbb R^d)\) is dense in
\(W^{2,2}(\mathbb R^d)\) and
\[
\mathcal N_d(u)
\le
\|u\|_{W^{2,2}(\mathbb R^d)},
\qquad
u\in C_c^\infty(\mathbb R^d),
\]
the identity map on \(C_c^\infty(\mathbb R^d)\) extends continuously to
an embedding
\[
W^{2,2}(\mathbb R^d)\hookrightarrow\mathcal X_d.
\]
Hence, by \eqref{variational-HN},
\[
\inf_{u\in\mathcal X_d\setminus\{0\}}
\frac{\sqrt{A(u)B(u)}}{J(u)}
\le
\mathcal H_d.
\]

Conversely, \eqref{eq:HyUP} holds on
\(C_c^\infty(\mathbb R^d)\), and the quantities \(A\), \(B\), and \(J\)
extend continuously to \(\mathcal X_d\). Therefore,
\[
\sqrt{A(u)B(u)}
\ge
\mathcal H_d J(u),
\qquad
u\in\mathcal X_d.
\]
Moreover, if \(J(u)=0\), then \(\nabla u=0\), and hence
\(A(u)=B(u)=0\), so \(u=0\) in \(\mathcal X_d\). Thus
\[
J(u)>0,
\qquad
u\in\mathcal X_d\setminus\{0\}.
\]
Consequently,
\[
\frac{\sqrt{A(u)B(u)}}{J(u)}
\ge
\mathcal H_d,
\qquad
u\in\mathcal X_d\setminus\{0\}.
\]
Taking the infimum yields the reverse inequality and proves
\eqref{Hd-Xd-representation}.
\end{proof}

We next derive an equivalent linearized formulation of
\eqref{Hd-Xd-representation}. Define
\begin{equation}\label{eq-Cd}
\mathcal C_d
=
\inf_{u\in\mathcal X_d\setminus\{0\}}
\frac{A(u)+B(u)}{J(u)}.
\end{equation}

\begin{proposition}\label{PR2.2}
The constants \(\mathcal H_d\) and \(\mathcal C_d\) satisfy
\begin{equation}\label{Cd-Hd-relation}
\mathcal C_d=2\mathcal H_d.
\end{equation}
\end{proposition}

\begin{proof}
Let \(u\in\mathcal X_d\setminus\{0\}\). For \(\lambda>0\), the dilation
\(u(x)\mapsto u(\lambda x)\), initially defined on
\(C_c^\infty(\mathbb R^d)\), extends continuously to \(\mathcal X_d\);
we denote the resulting element by \(u_\lambda\). We have
\[
A(u_\lambda)=\lambda^{4-d}A(u),
\qquad
B(u_\lambda)=\lambda^{2-d}B(u),
\qquad
J(u_\lambda)=\lambda^{3-d}J(u).
\]
Hence
\[
\frac{A(u_\lambda)+B(u_\lambda)}{J(u_\lambda)}
=
\frac{\lambda A(u)+\lambda^{-1}B(u)}{J(u)}.
\]
Since \(u\ne0\), the inequality
\[
\sqrt{A(u)B(u)}\ge \mathcal H_d J(u)>0
\]
implies that \(A(u)>0\) and \(B(u)>0\). The right-hand side is therefore
minimized at
\[
\lambda^*
=
\sqrt{\frac{B(u)}{A(u)}},
\]
and
\[
\inf_{\lambda>0}
\frac{A(u_\lambda)+B(u_\lambda)}{J(u_\lambda)}
=
2\frac{\sqrt{A(u)B(u)}}{J(u)}.
\]
Taking the infimum over \(u\in\mathcal X_d\setminus\{0\}\) and using
\eqref{Hd-Xd-representation} gives
\[
\mathcal C_d\le2\mathcal H_d.
\]

Conversely, we have
\[
\frac{A(u)+B(u)}{J(u)}
\ge
2\frac{\sqrt{A(u)B(u)}}{J(u)},
\qquad
u\in\mathcal X_d\setminus\{0\}.
\]
Taking the infimum and using \eqref{Hd-Xd-representation}, we obtain
\[
\mathcal C_d\ge2\mathcal H_d.
\]
Thus \eqref{Cd-Hd-relation} follows.
\end{proof}

In view of Proposition~\ref{PR2.2}, determining \(\mathcal H_d\) is
equivalent to determining the linearized constant \(\mathcal C_d\).
We next analyze \(\mathcal C_d\) by means of spherical harmonic
decomposition.

\section{Spherical harmonic reduction}\label{sec3}

Throughout this section, let \(d\ge2\). For \(k\in\mathbb N_0\), let
\(\mathcal Y_k(\mathbb S^{d-1})\) denote the space of real spherical
harmonics of degree \(k\) on \(\mathbb S^{d-1}\), and set
\[
m_{d,k}
=
\dim\mathcal Y_k(\mathbb S^{d-1}).
\]
For each \(k\in\mathbb N_0\), choose an
\(L^2(\mathbb S^{d-1})\)-orthonormal basis
\[
\{Y_{k,\ell}\}_{\ell=1}^{m_{d,k}}
\]
of \(\mathcal Y_k(\mathbb S^{d-1})\). We recall the following standard
properties; see, for instance, \cite{Muller1966}. For
\(k\in\mathbb N_0\) and \(1\le\ell\le m_{d,k}\),
\[
-\Delta_{\mathbb S^{d-1}}Y_{k,\ell}
=
k(k+d-2)Y_{k,\ell}.
\]
Moreover, for \(k,k'\in\mathbb N_0\),
\(1\le\ell\le m_{d,k}\), and
\(1\le\ell'\le m_{d,k'}\),
\[
\int_{\mathbb S^{d-1}}
Y_{k,\ell}(\omega)Y_{k',\ell'}(\omega)\,d\omega
=
\delta_{kk'}\delta_{\ell\ell'},
\]
where \(\delta_{ij}\) denotes the Kronecker delta. By integration by
parts on \(\mathbb S^{d-1}\),
\[
\int_{\mathbb S^{d-1}}
\nabla_{\mathbb S^{d-1}}Y_{k,\ell}(\omega)\cdot
\nabla_{\mathbb S^{d-1}}Y_{k',\ell'}(\omega)\,d\omega
=
k(k+d-2)\delta_{kk'}\delta_{\ell\ell'}.
\]
Finally,
\[
\left\{
Y_{k,\ell}:
k\in\mathbb N_0,\ 
1\le\ell\le m_{d,k}
\right\}
\]
is a complete orthonormal system in \(L^2(\mathbb S^{d-1})\).

Let \(u\in C_c^\infty(\mathbb R^d)\). By spherical harmonic
decomposition and the regularity at the origin of the corresponding
radial coefficients; see, for instance, \cite{DoLamLu2026}, \(u\) admits
the expansion
\begin{equation}\label{exp-u-SH}
u(r\omega)
=
\sum_{k=0}^{\infty}
\sum_{\ell=1}^{m_{d,k}}
r^k v_{k,\ell}(r)Y_{k,\ell}(\omega),
\qquad
r>0,\quad \omega\in\mathbb S^{d-1},
\end{equation}
where
\[
v_{k,\ell}\in C_c^\infty([0,\infty)),
\qquad
k\in\mathbb N_0,\quad
1\le\ell\le m_{d,k}.
\]

For \(k\in\mathbb N_0\), set
\[
\nu_{d,k}=d+2k.
\]
For a radial profile \(v\) for which the following integrals are finite,
define
\begin{equation}\label{def-Adk}
A_{d,k}(v)
=
\int_0^\infty
r^{\nu_{d,k}-1}|v''(r)|^2\,dr
+
(\nu_{d,k}-1)
\int_0^\infty
r^{\nu_{d,k}-3}|v'(r)|^2\,dr,
\end{equation}
and
\begin{equation}\label{def-Bdk}
B_{d,k}(v)
=
\int_0^\infty
r^{\nu_{d,k}-1}|v'(r)|^2\,dr.
\end{equation}
For \(k=0\), set
\[
J_{d,0}(v)
=
\int_0^\infty
r^{d-2}|v'(r)|^2\,dr,
\]
while, for \(k\in\mathbb N\), define
\begin{equation}\label{def-Jdk}
J_{d,k}(v)
=
\int_0^\infty
r^{\nu_{d,k}-2}|v'(r)|^2\,dr
+
k\int_0^\infty
r^{\nu_{d,k}-4}|v(r)|^2\,dr.
\end{equation}

\begin{lemma}\label{L3.1}
Let \(u\in C_c^\infty(\mathbb R^d)\), and let
\((v_{k,\ell})\) be the radial coefficients in the expansion
\eqref{exp-u-SH}. Then
\[
B(u)
=
\sum_{k=0}^\infty
\sum_{\ell=1}^{m_{d,k}}
B_{d,k}(v_{k,\ell}),
\]
\[
A(u)
=
\sum_{k=0}^\infty
\sum_{\ell=1}^{m_{d,k}}
A_{d,k}(v_{k,\ell}),
\]
and
\[
J(u)
=
\sum_{k=0}^\infty
\sum_{\ell=1}^{m_{d,k}}
J_{d,k}(v_{k,\ell}).
\]
\end{lemma}

\begin{proof}
The identities follow from the orthogonality of the spherical harmonics,
together with integration by parts in the radial variable; see also
\cite[Lemma~3.1]{CazacuFlynnLam}.
\end{proof}

For \(k\in\mathbb N_0\), let \(\mathscr D_{d,k}\) denote the set of all
radial profiles \(v_{k,\ell}\) arising in the spherical harmonic
expansion \eqref{exp-u-SH} as \(u\) ranges over
\(C_c^\infty(\mathbb R^d)\). We define the corresponding sectorial
constant by
\begin{equation}\label{def-Cdk}
\mathcal C_{d,k}
=
\inf_{v\in\mathscr D_{d,k}\setminus\{0\}}
\frac{A_{d,k}(v)+B_{d,k}(v)}{J_{d,k}(v)}.
\end{equation}

\begin{proposition}\label{P3.2}
The constant \(\mathcal C_d\) defined in \eqref{eq-Cd} satisfies
\begin{equation}\label{Cd-sector-reduction}
\mathcal C_d
=
\inf_{k\in\mathbb N_0}\mathcal C_{d,k}.
\end{equation}
\end{proposition}

\begin{proof}
By the density of \(C_c^\infty(\mathbb R^d)\) in \(\mathcal X_d\),
\[
\mathcal C_d
=
\inf_{u\in C_c^\infty(\mathbb R^d)\setminus\{0\}}
\frac{A(u)+B(u)}{J(u)}.
\]
Let \(u\in C_c^\infty(\mathbb R^d)\setminus\{0\}\). By
Lemma~\ref{L3.1} and \eqref{def-Cdk},
\[
\begin{aligned}
A(u)+B(u)
&=
\sum_{k=0}^\infty
\sum_{\ell=1}^{m_{d,k}}
\left(
A_{d,k}(v_{k,\ell})
+
B_{d,k}(v_{k,\ell})
\right)
\\
&\ge
\left(
\inf_{k\in\mathbb N_0}\mathcal C_{d,k}
\right)
\sum_{k=0}^\infty
\sum_{\ell=1}^{m_{d,k}}
J_{d,k}(v_{k,\ell})
\\
&=
\left(
\inf_{k\in\mathbb N_0}\mathcal C_{d,k}
\right)J(u).
\end{aligned}
\]
Hence
\[
\mathcal C_d
\ge
\inf_{k\in\mathbb N_0}\mathcal C_{d,k}.
\]

Conversely, let \(k\in\mathbb N_0\) and
\(v\in\mathscr D_{d,k}\setminus\{0\}\). By the definition of
\(\mathscr D_{d,k}\), \(v\) occurs as a coefficient of some
\(u\in C_c^\infty(\mathbb R^d)\). Taking the corresponding spherical
harmonic projection, for some \(1\le\ell\le m_{d,k}\) the function
\[
\widetilde u(r\omega)
=
r^k v(r)Y_{k,\ell}(\omega)
\]
belongs to \(C_c^\infty(\mathbb R^d)\). By Lemma~\ref{L3.1},
\[
A(\widetilde u)=A_{d,k}(v),
\qquad
B(\widetilde u)=B_{d,k}(v),
\qquad
J(\widetilde u)=J_{d,k}(v).
\]
Therefore,
\[
\mathcal C_d
\le
\frac{A_{d,k}(v)+B_{d,k}(v)}{J_{d,k}(v)}.
\]
Taking the infimum first over
\(v\in\mathscr D_{d,k}\setminus\{0\}\) and then over
\(k\in\mathbb N_0\) yields the reverse inequality.
\end{proof}

\begin{lemma}\label{L3.3}
The radial sector satisfies
\begin{equation}\label{Cd0-value}
\mathcal C_{d,0}=d+1.
\end{equation}
\end{lemma}

\begin{proof}
Let \(v\in\mathscr D_{d,0}\), and set \(g=v'\). A direct computation
gives
\[
A_{d,0}(v)+B_{d,0}(v)-(d+1)J_{d,0}(v)
=
\int_0^\infty
r^{d-3}
\left|
rg'(r)+(r-1)g(r)
\right|^2\,dr
\ge0.
\]
Hence
\[
\mathcal C_{d,0}\ge d+1.
\]

For the reverse inequality, let
\[
v_0(r)=(1+r)e^{-r},
\qquad r\ge0,
\]
and set
\[
u_0(x)=v_0(|x|)Y_{0,1}.
\]
Then \(u_0\in W^{2,2}(\mathbb R^d)\), and
\[
rv_0''(r)+(r-1)v_0'(r)=0,
\qquad r>0.
\]
Hence
\[
A_{d,0}(v_0)+B_{d,0}(v_0)
=
(d+1)J_{d,0}(v_0).
\]

By a standard radial approximation, there exists a sequence
\((u_n)\subset C_c^\infty(\mathbb R^d)\) of radial functions such that
\[
u_n\longrightarrow u_0
\qquad\text{in }W^{2,2}(\mathbb R^d)
\quad\text{as }n\to\infty.
\]
Writing
\[
u_n(x)=v_n(|x|)Y_{0,1},
\]
we have
\[
v_n\in\mathscr D_{d,0}.
\]
Moreover,
\[
A_{d,0}(v_n)\longrightarrow A_{d,0}(v_0),
\qquad
B_{d,0}(v_n)\longrightarrow B_{d,0}(v_0),
\qquad
\text{as }n\to\infty,
\]
and, by the continuity of \(J\) on \(\mathcal X_d\),
\[
J_{d,0}(v_n)\longrightarrow J_{d,0}(v_0)
\qquad\text{as }n\to\infty.
\]
Therefore,
\[
\frac{A_{d,0}(v_n)+B_{d,0}(v_n)}
{J_{d,0}(v_n)}
\longrightarrow d+1
\qquad\text{as }n\to\infty.
\]
Hence
\[
\mathcal C_{d,0}\le d+1,
\]
and \eqref{Cd0-value} follows.
\end{proof}

\begin{lemma}\label{L3.4}
For every \(k\ge2\),
\begin{equation}\label{Cdk-lower}
\mathcal C_{d,k}
\ge
\frac{
(d+2k+1)(d+2k-3)^2
}{
(d+2k-3)^2+4k
}
>
d+1.
\end{equation}
\end{lemma}

\begin{proof}
The first inequality follows from the sectorwise estimate established in
\cite[Proof of Theorem~2.3]{CazacuFlynnLam}. For the second one, observe
that
\[
\begin{aligned}
&(d+2k+1)(d+2k-3)^2
-(d+1)\left((d+2k-3)^2+4k\right)
\\
&=
2k\left((d+2k-3)^2-2(d+1)\right)>0,
\end{aligned}
\]
since \(d\ge2\) and \(k\ge2\). This proves \eqref{Cdk-lower}.
\end{proof}

\begin{proposition}\label{PR3.5}
We have
\begin{equation}\label{Cd-critical-sectors}
\mathcal C_d
=
\min\left\{
d+1,\mathcal C_{d,1}
\right\}.
\end{equation}
\end{proposition}

\begin{proof}
The conclusion follows immediately from
\eqref{Cd-sector-reduction}, \eqref{Cd0-value}, and
\eqref{Cdk-lower}.
\end{proof}

\section{One-dimensional reduction}\label{sec4}

Throughout this section, let \(d\ge2\). In view of
Proposition~\ref{PR3.5}, it remains to analyze the first non-radial
sector \(k=1\).

For \(y\in H_0^1(0,\infty)\), define
\begin{equation}\label{def-Ed}
E_d(y)
=
\int_0^\infty
\left(
|y'(r)|^2
+
\left(
1+\frac{(d+1)(d+3)}{4r^2}
\right)|y(r)|^2
\right)\,dr,
\end{equation}
\begin{equation}\label{def-D}
D(y)
=
\int_0^\infty
\frac{|y(r)|^2}{r}\,dr,
\end{equation}
and
\begin{equation}\label{def-Gd}
G_d(y)
=
\int_0^\infty
r^{d-2}
\left|
\int_r^\infty
\frac{y(s)}{s^{(d+1)/2}}\,ds
\right|^2\,dr.
\end{equation}

By Hardy's inequality and its weighted version, these quantities are
finite for every \(y\in H_0^1(0,\infty)\). Indeed,
\[
D(y)
\le
\|y\|_2
\left\|\frac{y}{r}\right\|_2
\le
2\|y\|_2\|y'\|_2
\le
\|y\|_{H^1(0,\infty)}^2,
\]
while the weighted Hardy inequality yields
\begin{equation}\label{Gd-bound}
G_d(y)
\le
\frac{4}{(d-1)^2}D(y).
\end{equation}
Moreover, Hardy's inequality gives
\[
\|y\|_{H^1(0,\infty)}^2
\le
E_d(y)
\le
(d+2)^2\|y\|_{H^1(0,\infty)}^2.
\]
Thus \(E_d^{1/2}\) defines a norm on \(H_0^1(0,\infty)\) equivalent to
the usual \(H^1(0,\infty)\)-norm. In particular, \(E_d\), \(D\), and
\(G_d\) are continuous quadratic forms on \(H_0^1(0,\infty)\).

For \(v\in\mathscr D_{d,1}\), define
\[
T_d:\mathscr D_{d,1}\longrightarrow H_0^1(0,\infty),
\qquad
(T_dv)(r)
=
r^{(d+1)/2}v'(r),
\qquad
r>0.
\]

\begin{lemma}\label{L4.1}
For every \(v\in\mathscr D_{d,1}\),
\begin{equation}\label{Td-energy}
E_d(T_dv)
=
A_{d,1}(v)+B_{d,1}(v),
\end{equation}
and
\begin{equation}\label{Td-J}
D(T_dv)+G_d(T_dv)
=
J_{d,1}(v),
\end{equation}
where \(A_{d,1}\), \(B_{d,1}\), and \(J_{d,1}\) are defined by
\eqref{def-Adk}, \eqref{def-Bdk}, and \eqref{def-Jdk},
respectively.
\end{lemma}

\begin{proof}
Let \(v\in\mathscr D_{d,1}\), and set
\[
y=T_dv=r^{(d+1)/2}v'.
\]
Using \eqref{def-Ed}, a direct computation and integration by parts,
with vanishing boundary terms, give \eqref{Td-energy}.

Moreover, by \eqref{def-D},
\[
D(y)
=
\int_0^\infty r^d|v'(r)|^2\,dr.
\]
Since \(v\) is compactly supported,
\[
\int_r^\infty
\frac{y(s)}{s^{(d+1)/2}}\,ds
=
\int_r^\infty v'(s)\,ds
=
-v(r).
\]
Therefore, by \eqref{def-Gd},
\[
G_d(y)
=
\int_0^\infty r^{d-2}|v(r)|^2\,dr.
\]
Thus \eqref{Td-J} follows from \eqref{def-Jdk}.
\end{proof}

\begin{remark}\label{RK4.2}
Observe that
\[
C_c^\infty(0,\infty)
\subset
T_d(\mathscr D_{d,1}).
\]
Indeed, let \(y\in C_c^\infty(0,\infty)\) and define
\[
v(r)
=
-\int_r^\infty
\frac{y(s)}{s^{(d+1)/2}}\,ds,
\qquad
r\ge0.
\]
Since \(y\) is compactly supported in \((0,\infty)\), the function \(v\)
is smooth, constant near \(0\), and compactly supported. Moreover,
\[
v'(r)
=
\frac{y(r)}{r^{(d+1)/2}},
\qquad
r>0,
\]
and hence
\[
T_dv=y.
\]
For any \(1\le\ell\le m_{d,1}\), the function
\[
u(r\omega)
=
rv(r)Y_{1,\ell}(\omega)
\]
belongs to \(C_c^\infty(\mathbb R^d)\). Therefore
\[
v\in\mathscr D_{d,1},
\]
which proves the inclusion.
\end{remark}

\begin{proposition}\label{PR4.3}
We have
\begin{equation}\label{Cd1-H01-representation}
\mathcal C_{d,1}
=
\inf_{y\in H_0^1(0,\infty)\setminus\{0\}}
\frac{E_d(y)}{D(y)+G_d(y)}.
\end{equation}
\end{proposition}

\begin{proof}
By \eqref{def-Cdk} and Lemma~\ref{L4.1},
\[
\mathcal C_{d,1}
=
\inf_{y\in T_d(\mathscr D_{d,1})\setminus\{0\}}
\frac{E_d(y)}{D(y)+G_d(y)}.
\]
Moreover, by Remark~\ref{RK4.2},
\[
C_c^\infty(0,\infty)
\subset
T_d(\mathscr D_{d,1})
\subset
H_0^1(0,\infty).
\]
Since \(C_c^\infty(0,\infty)\) is dense in \(H_0^1(0,\infty)\), and
\(E_d\), \(D\), and \(G_d\) are continuous quadratic forms on
\(H_0^1(0,\infty)\), the quotient
\[
y\longmapsto
\frac{E_d(y)}{D(y)+G_d(y)}
\]
is continuous on \(H_0^1(0,\infty)\setminus\{0\}\). Therefore,
\[
\inf_{y\in C_c^\infty(0,\infty)\setminus\{0\}}
\frac{E_d(y)}{D(y)+G_d(y)}
=
\inf_{y\in H_0^1(0,\infty)\setminus\{0\}}
\frac{E_d(y)}{D(y)+G_d(y)},
\]
and \eqref{Cd1-H01-representation} follows.
\end{proof}

For \(c>0\), define
\begin{equation}\label{def-Qdc}
Q_{d,c}(y)
=
E_d(y)-cD(y),
\qquad
y\in H_0^1(0,\infty).
\end{equation}

\begin{lemma}\label{L4.4}
Let \(c>0\). The following assertions are equivalent:
\begin{itemize}
\item[(i)] \(\mathcal C_{d,1}\ge c\).
\item[(ii)] \(Q_{d,c}(y)\ge cG_d(y)\) for every
\(y\in H_0^1(0,\infty)\).
\end{itemize}
\end{lemma}

\begin{proof}
The equivalence follows immediately from
\eqref{Cd1-H01-representation} and \eqref{def-Qdc}.
\end{proof}

\section{Spectral decomposition and discrete reduction}\label{sec5}

Throughout this section, let \(d\ge2\). By Lemma~\ref{L4.4},
\(\mathcal C_{d,1}\) is the largest \(c>0\) for which
\[
Q_{d,c}(y)\ge c\,G_d(y),
\qquad
y\in H_0^1(0,\infty).
\]
We introduce a Laguerre-type spectral basis adapted to the quadratic
forms \(E_d\) and \(D\). Expansion with respect to this basis
simultaneously diagonalizes \(E_d\) and \(D\), and hence \(Q_{d,c}\),
while \(G_d\) is represented by a quadratic form on
\(\ell^2(\mathbb N_0)\). This yields an equivalent discrete operator
formulation and leads to the spectral characterization of
\(\mathcal H_d\). We conclude the section with the proof of
Theorem~\ref{thm:spectral-characterization} and its consequences,
Corollaries~\ref{cor:Hd-dge4} and~\ref{cor:Hd-d23}.

\subsection{Laguerre spectral basis}

We recall that, for \(\alpha>-1\), the generalized Laguerre polynomials
\(L_n^{(\alpha)}\), \(n\in\mathbb N_0\), satisfy
\begin{equation}\label{eq:Laguerre-ODE}
x\bigl(L_n^{(\alpha)}\bigr)''(x)
+
(\alpha+1-x)\bigl(L_n^{(\alpha)}\bigr)'(x)
+
nL_n^{(\alpha)}(x)
=
0,
\qquad
x>0,
\end{equation}
and the orthogonality relation
\begin{equation}\label{eq:Laguerre-orthogonality}
\int_0^\infty
x^\alpha e^{-x}
L_n^{(\alpha)}(x)L_m^{(\alpha)}(x)\,dx
=
\frac{\Gamma(n+\alpha+1)}{\Gamma(n+1)}
\delta_{nm},
\qquad
n,m\in\mathbb N_0.
\end{equation}
Moreover, the family
\(\{L_n^{(\alpha)}\}_{n\ge0}\) is complete in
\[
L^2\bigl((0,\infty),x^\alpha e^{-x}\,dx\bigr);
\]
see, for instance, \cite{Szego1975}.

For \(n\in\mathbb N_0\), define
\[
\psi_n^{(d)}(r)
=
r^{(d+3)/2}e^{-r}L_n^{(d+2)}(2r),
\qquad
r>0.
\]
Then
\[
\psi_n^{(d)}\in H_0^1(0,\infty),
\qquad
n\in\mathbb N_0.
\]
By \eqref{eq:Laguerre-ODE},
\begin{equation}\label{L-ode-d}
-\bigl(\psi_n^{(d)}\bigr)''(r)
+
\left(
1+\frac{(d+1)(d+3)}{4r^2}
\right)\psi_n^{(d)}(r)
=
\frac{\lambda_n^{(d)}}{r}\psi_n^{(d)}(r),
\qquad
r>0,
\end{equation}
where
\begin{equation}\label{def-lambda-nd}
\lambda_n^{(d)}
=
2n+d+3.
\end{equation}

Furthermore, by the change of variables \(x=2r\) and
\eqref{eq:Laguerre-orthogonality},
\begin{equation}\label{psi-d-orthogonality}
\int_0^\infty
\frac{\psi_n^{(d)}(r)\psi_m^{(d)}(r)}{r}\,dr
=
\omega_n^{(d)}\delta_{nm},
\qquad
n,m\in\mathbb N_0,
\end{equation}
where
\begin{equation}\label{def-omega-nd}
\omega_n^{(d)}
=
\frac{
\Gamma(n+d+3)
}{
2^{d+3}\Gamma(n+1)
}.
\end{equation}
Finally, the completeness of the generalized Laguerre polynomials implies
that
\begin{equation}\label{psi-d-completeness}
\left\{\psi_n^{(d)}\right\}_{n\ge0}
\quad\text{is complete in}\quad
L^2\left((0,\infty),\frac{dr}{r}\right).
\end{equation}

For \(y\in H_0^1(0,\infty)\), we have
\[
y\in
L^2\left((0,\infty),\frac{dr}{r}\right).
\]
Hence, by \eqref{psi-d-orthogonality} and
\eqref{psi-d-completeness},
\begin{equation}\label{exp-y-psi-d}
y
=
\sum_{n=0}^\infty a_n\psi_n^{(d)}
\qquad\text{in }
L^2\left((0,\infty),\frac{dr}{r}\right),
\end{equation}
where
\begin{equation}\label{def-an-d}
a_n
=
\frac{1}{\omega_n^{(d)}}
\int_0^\infty
\frac{y(r)\psi_n^{(d)}(r)}{r}\,dr,
\qquad
n\in\mathbb N_0.
\end{equation}

\begin{lemma}
Let \(y\in H_0^1(0,\infty)\), and let \((a_n)_{n\ge0}\) be defined by
\eqref{def-an-d}. Then
\begin{align}
D(y)
&=
\sum_{n=0}^\infty
\omega_n^{(d)}|a_n|^2,
\label{D-series-d}
\\
E_d(y)
&=
\sum_{n=0}^\infty
\lambda_n^{(d)}
\omega_n^{(d)}|a_n|^2.
\label{Ed-series}
\end{align}
Moreover, for every \(c>0\),
\begin{equation}\label{Qdc-series}
Q_{d,c}(y)
=
\sum_{n=0}^\infty
\bigl(\lambda_n^{(d)}-c\bigr)
\omega_n^{(d)}|a_n|^2.
\end{equation}
\end{lemma}

\begin{proof}
By \eqref{def-D}, \eqref{psi-d-orthogonality},
\eqref{psi-d-completeness}, and Parseval's identity,
\eqref{D-series-d} follows.

Define
\[
\mathcal E_d(y,z)
=
\int_0^\infty
\left(
y'(r)z'(r)
+
\left(
1+\frac{(d+1)(d+3)}{4r^2}
\right)y(r)z(r)
\right)\,dr,
\qquad
y,z\in H_0^1(0,\infty).
\]
Then
\[
E_d(y)=\mathcal E_d(y,y),
\qquad
y\in H_0^1(0,\infty).
\]
Since \(E_d^{1/2}\) is equivalent to the usual
\(H^1(0,\infty)\)-norm, \(H_0^1(0,\infty)\), endowed with the inner
product \(\mathcal E_d\), is a Hilbert space.

By \eqref{L-ode-d}, \eqref{psi-d-orthogonality}, and integration by parts,
\[
\mathcal E_d
\left(
\psi_n^{(d)},\psi_m^{(d)}
\right)
=
\lambda_n^{(d)}
\omega_n^{(d)}
\delta_{nm},
\qquad
n,m\in\mathbb N_0.
\]
Moreover, the system
\[
\left\{
\frac{\psi_n^{(d)}}
{\sqrt{\lambda_n^{(d)}\omega_n^{(d)}}}
\right\}_{n\ge0}
\]
is complete in \(H_0^1(0,\infty)\) with respect to
\(\mathcal E_d\). Indeed, if \(y\in H_0^1(0,\infty)\) satisfies
\[
\mathcal E_d\left(y,\psi_n^{(d)}\right)=0,
\qquad
n\in\mathbb N_0,
\]
then \eqref{L-ode-d} and integration by parts give
\[
\lambda_n^{(d)}
\int_0^\infty
\frac{y(r)\psi_n^{(d)}(r)}{r}\,dr
=
0,
\qquad
n\in\mathbb N_0.
\]
Since \(\lambda_n^{(d)}>0\), \eqref{psi-d-completeness} yields \(y=0\).

Furthermore, by \eqref{L-ode-d} and \eqref{def-an-d},
\[
\mathcal E_d\left(y,\psi_n^{(d)}\right)
=
\lambda_n^{(d)}
\int_0^\infty
\frac{y(r)\psi_n^{(d)}(r)}{r}\,dr
=
\lambda_n^{(d)}\omega_n^{(d)}a_n.
\]
Parseval's identity in the Hilbert space
\(\bigl(H_0^1(0,\infty),\mathcal E_d\bigr)\) therefore gives
\eqref{Ed-series}. Finally, \eqref{Qdc-series} follows from
\eqref{def-Qdc}, \eqref{D-series-d}, and \eqref{Ed-series}.
\end{proof}

For \(n\in\mathbb N_0\), define
\begin{equation}\label{def-Phi-nd}
\Phi_n^{(d)}(r)
=
\int_r^\infty
s e^{-s}L_n^{(d+2)}(2s)\,ds,
\qquad
r>0,
\end{equation}
and, for \(n,m\in\mathbb N_0\), set
\begin{equation}\label{def-g-nm-d}
g_{n,m}^{(d)}
=
\int_0^\infty
r^{d-2}
\Phi_n^{(d)}(r)\Phi_m^{(d)}(r)\,dr.
\end{equation}

\begin{lemma}
Let \(y\in H_0^1(0,\infty)\), and let \((a_n)_{n\ge0}\) be the
coefficients defined by \eqref{def-an-d}. Then
\begin{equation}\label{Gd-series}
G_d(y)
=
\lim_{N\to\infty}
\sum_{n,m=0}^N
g_{n,m}^{(d)}a_na_m.
\end{equation}
\end{lemma}

\begin{proof}
For \(N\in\mathbb N_0\), set
\[
y_N
=
\sum_{n=0}^N a_n\psi_n^{(d)}.
\]
By  \eqref{def-D} and \eqref{exp-y-psi-d},
\[
D(y_N-y)\longrightarrow0
\qquad\text{as }N\to\infty.
\]
Hence, by \eqref{Gd-bound},
\[
G_d(y_N-y)\longrightarrow0
\qquad\text{as }N\to\infty.
\]
Moreover, by \eqref{def-Gd} and the reverse triangle inequality in
\(L^2((0,\infty),r^{d-2}\,dr)\),
\[
\left|
G_d(y_N)^{1/2}
-
G_d(y)^{1/2}
\right|
\le
G_d(y_N-y)^{1/2}.
\]
Therefore,
\[
G_d(y_N)\longrightarrow G_d(y)
\qquad\text{as }N\to\infty.
\]

On the other hand, by \eqref{def-Phi-nd},
\[
\int_r^\infty
\frac{y_N(s)}{s^{(d+1)/2}}\,ds
=
\sum_{n=0}^N a_n\Phi_n^{(d)}(r).
\]
Thus, by \eqref{def-Gd} and \eqref{def-g-nm-d},
\[
G_d(y_N)
=
\sum_{n,m=0}^N
g_{n,m}^{(d)}a_na_m.
\]
Letting \(N\to\infty\) and using the convergence of \(G_d(y_N)\)
established above yields \eqref{Gd-series}.
\end{proof}

For \(j\in\mathbb N_0\), set
\begin{equation}\label{def-hhat-jd}
\widehat h_j^{(d)}
=
\frac{(d-1)^2}{2^{d+1}}
\sum_{m=0}^j
\frac{
\Gamma(2m+d+2)
}{
\Gamma(2m+2)(4m+d+1)(4m+d+5)
}.
\end{equation}

\begin{lemma}
For all \(j,k\in\mathbb N_0\),
\begin{equation}\label{g-even-odd}
g_{2j,2k+1}^{(d)}
=
g_{2k+1,2j}^{(d)}
=
0,
\end{equation}
\begin{equation}\label{g-even-even}
g_{2j,2k}^{(d)}
=
\frac{
(4j+d+3)(4k+d+3)
}{
(d-1)^2
}
h_{\min\{j,k\}}^{(d)},
\end{equation}
where \(h_j^{(d)}\) is defined by \eqref{eq-hj-gammaj}, and
\begin{equation}\label{g-odd-odd}
g_{2j+1,2k+1}^{(d)}
=
\frac{
(4j+d+5)(4k+d+5)
}{
(d-1)^2
}
\widehat h_{\min\{j,k\}}^{(d)}.
\end{equation}
\end{lemma}

\begin{proof}
Using the classical generating function for the generalized Laguerre
polynomials (see, e.g., \cite{Szego1975}),
\[
\sum_{n=0}^{\infty}L_n^{(\alpha)}(x)z^n
=
\frac{1}{(1-z)^{\alpha+1}}
\exp\left(-\frac{xz}{1-z}\right),
\qquad |z|<1,
\]
with \(\alpha=d+2\) and \(x=2r\), we obtain
\[
\sum_{n=0}^{\infty}L_n^{(d+2)}(2r)z^n
=
\frac{1}{(1-z)^{d+3}}
\exp\left(-\frac{2rz}{1-z}\right).
\]
Hence, using \eqref{def-Phi-nd} and integrating term by term,
\begin{equation}\label{gf-Phi-d}
\sum_{n=0}^\infty \Phi_n^{(d)}(r)z^n
=
\frac{(1-z)+r(1+z)}
{(1-z)^{d+2}(1+z)^2}
\exp\left(
-\frac{1+z}{1-z}r
\right),
\qquad
r>0,\quad |z|<1.
\end{equation}
Using \eqref{def-g-nm-d} and \eqref{gf-Phi-d}, a direct computation gives,
for \(|z|,|w|<1\),
\begin{equation}\label{gf-g-d}
\begin{aligned}
\sum_{n,m=0}^\infty g_{n,m}^{(d)}z^nw^m
={}&
\frac{
d\,\Gamma(d-1)
}{
2^{d-1}
(1-z^2)^2(1-w^2)^2(1-zw)^{d-1}
}
\\
&+
\frac{
\Gamma(d+1)
}{
2^{d+1}
(1-z^2)(1-w^2)(1-zw)^{d+1}
}.
\end{aligned}
\end{equation}
The right-hand side of \eqref{gf-g-d} is invariant under
\[
(z,w)\longmapsto(-z,-w).
\]
Therefore,
\[
g_{n,m}^{(d)}
=
(-1)^{n+m}g_{n,m}^{(d)},
\qquad
n,m\in\mathbb N_0.
\]
In particular,
\[
g_{n,m}^{(d)}=0
\qquad
\text{whenever }n+m\text{ is odd},
\]
which proves \eqref{g-even-odd}.

We next prove \eqref{g-even-even}. By symmetry, it is enough to assume
\(j\le k\). Using
\[
\frac{1}{(1-z^2)^2}
=
\sum_{p=0}^\infty (p+1)z^{2p},
\qquad
\frac{1}{1-z^2}
=
\sum_{p=0}^\infty z^{2p},
\]
and
\[
\frac{1}{(1-zw)^\alpha}
=
\sum_{q=0}^\infty
\frac{\Gamma(q+\alpha)}
{\Gamma(\alpha)\Gamma(q+1)}
(zw)^q,
\qquad
\alpha\in\{d-1,d+1\},
\]
the Cauchy product of these absolutely convergent series, followed by
comparison of the coefficients of \(z^{2j}w^{2k}\) in
\eqref{gf-g-d}, gives
\begin{equation}\label{gd-even-even}
g_{2j,2k}^{(d)}
=
\frac{d}{2^{d-1}}
\sum_{m=0}^j
(j-m+1)(k-m+1)
\frac{\Gamma(2m+d-1)}{\Gamma(2m+1)}
+
\frac{1}{2^{d+1}}
\sum_{m=0}^j
\frac{\Gamma(2m+d+1)}{\Gamma(2m+1)}.
\end{equation}
By \eqref{eq-hj-gammaj},
\[
\begin{aligned}
&\frac{
(4j+d+3)(4k+d+3)
}{
(d-1)^2
}
h_j^{(d)}
\\
&\qquad=
\frac{
(4j+d+3)(4k+d+3)
}{
2^{d+1}
}
\sum_{m=0}^j
\frac{
\Gamma(2m+d+1)
}{
\Gamma(2m+1)(4m+d-1)(4m+d+3)
}.
\end{aligned}
\]
Set
\[
\rho_m
=
\frac{\Gamma(2m+d-1)}{\Gamma(2m+1)},
\qquad
m\in\mathbb N_0.
\]
Using \eqref{gd-even-even} and
\[
\Gamma(2m+d+1)
=
(2m+d)(2m+d-1)\Gamma(2m+d-1),
\]
we obtain
\begin{equation}\label{eq:g2j-2k}
g_{2j,2k}^{(d)}
-
\frac{
(4j+d+3)(4k+d+3)
}{
(d-1)^2
}
h_j^{(d)}
=
\frac{1}{2^{d+1}}\sum_{m=0}^j \Theta_m,
\end{equation}
where
\[
\Theta_m
=
\rho_m
\Bigg[
\frac{
4(2m+d)(2m+d-1)(j-m)(k-m)
}{
4m+d+3
}
-
\frac{
8m(2m-1)(j-m+1)(k-m+1)
}{
4m+d-1
}
\Bigg].
\]
Now set
\[
F_m
=
\frac{
8m(2m-1)(j-m+1)(k-m+1)
}{
4m+d-1
}
\rho_m,
\qquad
0\le m\le j+1.
\]
Since
\[
\rho_{m+1}
=
\frac{
(2m+d)(2m+d-1)
}{
(2m+2)(2m+1)
}
\rho_m,
\]
we obtain
\[
\Theta_m
=
F_{m+1}-F_m,
\qquad
0\le m\le j.
\]
Therefore,
\[
\sum_{m=0}^j\Theta_m
=
F_{j+1}-F_0
=
0,
\]
because
\[
F_0=F_{j+1}=0.
\]
Together with \eqref{eq:g2j-2k}, this yields
\[
g_{2j,2k}^{(d)}
=
\frac{
(4j+d+3)(4k+d+3)
}{
(d-1)^2
}
h_j^{(d)}.
\]
Since \(j\le k\), we have
\[
h_j^{(d)}
=
h_{\min\{j,k\}}^{(d)},
\]
and hence \eqref{g-even-even} follows.

We finally prove \eqref{g-odd-odd}. By symmetry, it is enough to assume
\(j\le k\). Using the same series expansions as above, the Cauchy
product, followed by comparison of the coefficients of
\(z^{2j+1}w^{2k+1}\) in \eqref{gf-g-d}, gives
\begin{equation}\label{gd-odd-odd}
g_{2j+1,2k+1}^{(d)}
=
\frac{d}{2^{d-1}}
\sum_{m=0}^j
(j-m+1)(k-m+1)
\frac{\Gamma(2m+d)}{\Gamma(2m+2)}
+
\frac{1}{2^{d+1}}
\sum_{m=0}^j
\frac{\Gamma(2m+d+2)}{\Gamma(2m+2)}.
\end{equation}
By \eqref{def-hhat-jd},
\[
\begin{aligned}
&\frac{
(4j+d+5)(4k+d+5)
}{
(d-1)^2
}
\widehat h_j^{(d)}
\\
&\qquad=
\frac{
(4j+d+5)(4k+d+5)
}{
2^{d+1}
}
\sum_{m=0}^j
\frac{
\Gamma(2m+d+2)
}{
\Gamma(2m+2)(4m+d+1)(4m+d+5)
}.
\end{aligned}
\]
Set
\[
\sigma_m
=
\frac{\Gamma(2m+d)}{\Gamma(2m+2)},
\qquad
m\in\mathbb N_0.
\]
Using \eqref{gd-odd-odd} and
\[
\Gamma(2m+d+2)
=
(2m+d+1)(2m+d)\Gamma(2m+d),
\]
we obtain
\begin{equation}\label{eq:g2j1-2k1}
g_{2j+1,2k+1}^{(d)}
-
\frac{
(4j+d+5)(4k+d+5)
}{
(d-1)^2
}
\widehat h_j^{(d)}
=
\frac{1}{2^{d+1}}
\sum_{m=0}^j\Phi_m,
\end{equation}
where
\[
\Phi_m
=
\sigma_m
\Bigg[
\frac{
4(2m+d+1)(2m+d)(j-m)(k-m)
}{
4m+d+5
}
-
\frac{
8m(2m+1)(j-m+1)(k-m+1)
}{
4m+d+1
}
\Bigg].
\]
Now set
\[
Z_m
=
\frac{
8m(2m+1)(j-m+1)(k-m+1)
}{
4m+d+1
}
\sigma_m,
\qquad
0\le m\le j+1.
\]
Since
\[
\sigma_{m+1}
=
\frac{
(2m+d+1)(2m+d)
}{
(2m+3)(2m+2)
}
\sigma_m,
\]
we obtain
\[
\Phi_m
=
Z_{m+1}-Z_m,
\qquad
0\le m\le j.
\]
Since
\[
Z_0=Z_{j+1}=0,
\]
\eqref{eq:g2j1-2k1} yields
\[
g_{2j+1,2k+1}^{(d)}
=
\frac{
(4j+d+5)(4k+d+5)
}{
(d-1)^2
}
\widehat h_j^{(d)}
=
\frac{
(4j+d+5)(4k+d+5)
}{
(d-1)^2
}
\widehat h_{\min\{j,k\}}^{(d)},
\]
and hence \eqref{g-odd-odd} follows.
\end{proof}

For \(y\in H_0^1(0,\infty)\), with expansion \eqref{exp-y-psi-d}, set
\[
y_{\mathrm e}
=
\sum_{j=0}^\infty
a_{2j}\psi_{2j}^{(d)},
\qquad
y_{\mathrm o}
=
\sum_{j=0}^\infty
a_{2j+1}\psi_{2j+1}^{(d)}.
\]

\begin{lemma}
The series defining \(y_{\mathrm e}\) and \(y_{\mathrm o}\) converge in
\(H_0^1(0,\infty)\), and
\begin{equation}\label{y-even-odd-decomposition}
y=y_{\mathrm e}+y_{\mathrm o}.
\end{equation}
Moreover, for every \(c>0\),
\begin{equation}\label{Qdc-even-odd}
Q_{d,c}(y)
=
Q_{d,c}(y_{\mathrm e})
+
Q_{d,c}(y_{\mathrm o}),
\end{equation}
and
\begin{equation}\label{Gd-even-odd}
G_d(y)
=
G_d(y_{\mathrm e})
+
G_d(y_{\mathrm o}).
\end{equation}
\end{lemma}

\begin{proof}
By \eqref{Ed-series},
\[
E_d(y)
=
\sum_{n=0}^\infty
\lambda_n^{(d)}\omega_n^{(d)}|a_n|^2
<
\infty.
\]
Since \(E_d^{1/2}\) is an equivalent norm on \(H_0^1(0,\infty)\), the
even and odd subseries converge in \(H_0^1(0,\infty)\). Their sum is the
full expansion of \(y\), and hence \eqref{y-even-odd-decomposition}
follows. Moreover, \eqref{Qdc-series} gives
\[
Q_{d,c}(y)
=
Q_{d,c}(y_{\mathrm e})
+
Q_{d,c}(y_{\mathrm o}),
\]
which proves \eqref{Qdc-even-odd}.

On the other hand, by \eqref{g-even-odd}, all mixed even--odd terms
vanish in the finite sums appearing in \eqref{Gd-series}. Passing to the
limit as \(N\to\infty\) and applying \eqref{Gd-series} to
\(y_{\mathrm e}\) and \(y_{\mathrm o}\), we obtain
\eqref{Gd-even-odd}.
\end{proof}

\subsection{The odd component}

For \(0<c<d+5\) and \(j\in\mathbb N_0\), set
\begin{equation}\label{def-gammahat-jd}
\widehat\gamma_j^{(d)}(c)
=
\frac{
2^{d+3}c(4j+d+5)^2\Gamma(2j+2)
}{
(d-1)^2(4j+d+5-c)\Gamma(2j+d+4)
}.
\end{equation}
For \(j,k\in\mathbb N_0\), define
\begin{equation}\label{def-shat-jk}
\widehat s_{j,k}^{(d)}(c)
=
\sqrt{\widehat\gamma_j^{(d)}(c)}
\,\widehat h_{\min\{j,k\}}^{(d)}
\sqrt{\widehat\gamma_k^{(d)}(c)},
\end{equation}
where \(\widehat h_j^{(d)}\) is defined by \eqref{def-hhat-jd}.

\begin{lemma}\label{L5.5}
Let \(0<c<d+5\), and let \(y\in H_0^1(0,\infty)\) have the expansion
\eqref{exp-y-psi-d}. Let
\[
y_{\mathrm o}
=
\sum_{j=0}^\infty
a_{2j+1}\psi_{2j+1}^{(d)}
\]
be its odd component, and set
\begin{equation}\label{def-bj-odd}
b_j
=
\sqrt{
(4j+d+5-c)\omega_{2j+1}^{(d)}
}\,a_{2j+1},
\qquad
j\in\mathbb N_0.
\end{equation}
Then \(b=(b_j)_{j\in\mathbb N_0}\in\ell^2(\mathbb N_0)\), and
\begin{equation}\label{Qdc-odd-l2}
Q_{d,c}(y_{\mathrm o})
=
\sum_{j=0}^\infty |b_j|^2.
\end{equation}
Moreover,
\begin{equation}\label{Gd-odd-l2}
cG_d(y_{\mathrm o})
=
\lim_{N\to\infty}
\sum_{j,k=0}^N
\widehat s_{j,k}^{(d)}(c)b_jb_k.
\end{equation}
\end{lemma}

\begin{proof}
By \eqref{def-lambda-nd} and \eqref{Qdc-series},
\[
Q_{d,c}(y_{\mathrm o})
=
\sum_{j=0}^\infty
(4j+d+5-c)\omega_{2j+1}^{(d)}
|a_{2j+1}|^2.
\]
Since \(0<c<d+5\), all the coefficients in this series are positive.
Thus, by \eqref{def-bj-odd},
\[
b\in\ell^2(\mathbb N_0),
\]
and \eqref{Qdc-odd-l2} follows.

On the other hand, applying \eqref{Gd-series} to \(y_{\mathrm o}\) and
using \eqref{g-odd-odd}, we obtain
\[
cG_d(y_{\mathrm o})
=
\lim_{N\to\infty}
\sum_{j,k=0}^N
c\,g_{2j+1,2k+1}^{(d)}
a_{2j+1}a_{2k+1}.
\]
By \eqref{def-bj-odd}, for \(0\le j,k\le N\),
\[
\begin{aligned}
c\,g_{2j+1,2k+1}^{(d)}
a_{2j+1}a_{2k+1}
&=
\frac{
c\,g_{2j+1,2k+1}^{(d)}b_jb_k
}{
\sqrt{
(4j+d+5-c)\omega_{2j+1}^{(d)}
}
\sqrt{
(4k+d+5-c)\omega_{2k+1}^{(d)}
}
}
\\
&=
\sqrt{\widehat\gamma_j^{(d)}(c)}
\,\widehat h_{\min\{j,k\}}^{(d)}
\sqrt{\widehat\gamma_k^{(d)}(c)}
\,b_jb_k
\\
&=
\widehat s_{j,k}^{(d)}(c)b_jb_k,
\end{aligned}
\]
where we used \eqref{def-omega-nd}, \eqref{g-odd-odd},
\eqref{def-gammahat-jd}, and \eqref{def-shat-jk}.
Therefore, \eqref{Gd-odd-l2} follows.
\end{proof}

\begin{lemma}\label{L5.6}
For every \(N\in\mathbb N_0\) and every
\((b_0,\ldots,b_N)\in\mathbb R^{N+1}\),
\begin{equation}\label{odd-form-bound-critical}
0
\le
\sum_{j,k=0}^N
\widehat s_{j,k}^{(d)}(d+1)b_jb_k
\le
\frac1d
\sum_{j=0}^N |b_j|^2.
\end{equation}
Consequently, for every \(0<c\le d+1\),
\begin{equation}\label{odd-form-bound}
0
\le
\sum_{j,k=0}^N
\widehat s_{j,k}^{(d)}(c)b_jb_k
\le
\frac1d
\sum_{j=0}^N |b_j|^2.
\end{equation}
\end{lemma}

\begin{proof}
We first prove \eqref{odd-form-bound-critical}. Set
\[
\widehat h_{-1}^{(d)}=0
\]
and
\[
\delta_m^{(d)}
=
\widehat h_m^{(d)}-\widehat h_{m-1}^{(d)},
\qquad
m\in\mathbb N_0.
\]
By \eqref{def-hhat-jd},
\begin{equation}\label{delta-md}
\delta_m^{(d)}
=
\frac{(d-1)^2}{2^{d+1}}
\frac{
\Gamma(2m+d+2)
}{
\Gamma(2m+2)(4m+d+1)(4m+d+5)
}
>0.
\end{equation}
Hence, for every \(j,k\in\mathbb N_0\),
\begin{equation}\label{htilde-min-series}
\widehat h_{\min\{j,k\}}^{(d)}
=
\sum_{m=0}^{\min\{j,k\}}
\delta_m^{(d)},
\end{equation}
and therefore, by \eqref{def-shat-jk},
\[
\begin{aligned}
\sum_{j,k=0}^N
\widehat s_{j,k}^{(d)}(d+1)b_jb_k
&=
\sum_{m=0}^N
\delta_m^{(d)}
\left(
\sum_{j=m}^N
\sqrt{\widehat\gamma_j^{(d)}(d+1)}\,b_j
\right)^2
\\
&\ge0.
\end{aligned}
\]

To prove the upper bound, define
\begin{equation}\label{pj-def}
p_j
=
\frac{j+1}{4j+d+5}
\sqrt{\widehat\gamma_j^{(d)}(d+1)},
\qquad
j\in\mathbb N_0.
\end{equation}
We claim that
\begin{equation}\label{Shat-Schur-identity}
\sum_{k=0}^\infty
\widehat s_{j,k}^{(d)}(d+1)p_k
=
\frac1d p_j,
\qquad
j\in\mathbb N_0.
\end{equation}
Indeed, by \eqref{def-gammahat-jd} and \eqref{pj-def},
\[
\sqrt{\widehat\gamma_k^{(d)}(d+1)}\,p_k
=
\frac{
2^{d+1}(d+1)
}{
(d-1)^2
}
(4k+d+5)
\frac{\Gamma(2k+2)}{\Gamma(2k+d+4)},
\qquad
k\in\mathbb N_0.
\]
Set
\[
r_k
=
\frac{\Gamma(2k+2)}{\Gamma(2k+d+2)},
\qquad
k\in\mathbb N_0.
\]
Using \eqref{eq-gamma-functional-relation}, we obtain
\[
r_k-r_{k+1}
=
d(4k+d+5)
\frac{\Gamma(2k+2)}{\Gamma(2k+d+4)}.
\]
Hence
\[
\sqrt{\widehat\gamma_k^{(d)}(d+1)}\,p_k
=
\frac{
2^{d+1}(d+1)
}{
d(d-1)^2
}
(r_k-r_{k+1}).
\]
On the other hand, by Stirling's formula,
\[
r_k\longrightarrow0
\qquad
\text{as }k\to\infty.
\]
Therefore,
\[
\begin{aligned}
\sum_{k=m}^\infty
\sqrt{\widehat\gamma_k^{(d)}(d+1)}\,p_k
&=
\frac{
2^{d+1}(d+1)
}{
d(d-1)^2
}
r_m
\\
&=
\frac{
2^{d+1}(d+1)
}{
d(d-1)^2
}
\frac{\Gamma(2m+2)}{\Gamma(2m+d+2)}.
\end{aligned}
\]
Thus, by \eqref{delta-md}, \eqref{htilde-min-series}, and \eqref{pj-def},
\[
\begin{aligned}
\sum_{k=0}^\infty
\widehat h_{\min\{j,k\}}^{(d)}
\sqrt{\widehat\gamma_k^{(d)}(d+1)}\,p_k
&=
\sum_{m=0}^j
\delta_m^{(d)}
\sum_{k=m}^\infty
\sqrt{\widehat\gamma_k^{(d)}(d+1)}\,p_k
\\
&=
\frac{d+1}{d}
\sum_{m=0}^j
\frac{1}{(4m+d+1)(4m+d+5)}
\\
&=
\frac{d+1}{4d}
\left(
\frac{1}{d+1}
-
\frac{1}{4j+d+5}
\right)
\\
&=
\frac{j+1}{d(4j+d+5)}
\\
&=
\frac{1}{\sqrt{\widehat\gamma_j^{(d)}(d+1)}}
\frac1d p_j.
\end{aligned}
\]
Multiplying by
\(\sqrt{\widehat\gamma_j^{(d)}(d+1)}\) and using
\eqref{def-shat-jk}, we obtain \eqref{Shat-Schur-identity}.

Since \(p_j>0\), Young's inequality gives
\[
2|b_jb_k|
\le
\frac{p_k}{p_j}|b_j|^2
+
\frac{p_j}{p_k}|b_k|^2.
\]
Using the symmetry and the nonnegativity of the entries
\(\widehat s_{j,k}^{(d)}(d+1)\), followed by
\eqref{Shat-Schur-identity}, we obtain
\[
\begin{aligned}
\sum_{j,k=0}^N
\widehat s_{j,k}^{(d)}(d+1)b_jb_k
&\le
\sum_{j,k=0}^N
\widehat s_{j,k}^{(d)}(d+1)|b_jb_k|
\\
&\le
\sum_{j=0}^N
\frac{|b_j|^2}{p_j}
\sum_{k=0}^N
\widehat s_{j,k}^{(d)}(d+1)p_k
\\
&\le
\frac1d
\sum_{j=0}^N|b_j|^2.
\end{aligned}
\]
This proves \eqref{odd-form-bound-critical}.

Now let \(0<c\le d+1\), and set
\[
\eta_j(c)
=
\left(
\frac{
\widehat\gamma_j^{(d)}(c)
}{
\widehat\gamma_j^{(d)}(d+1)
}
\right)^{1/2},
\qquad
j\in\mathbb N_0.
\]
Since
\[
t\longmapsto
\frac{t}{4j+d+5-t}
\]
is increasing on \((0,d+5)\), it follows from
\eqref{def-gammahat-jd} that
\[
t\longmapsto\widehat\gamma_j^{(d)}(t)
\]
is increasing on \((0,d+5)\). Hence
\[
0<\eta_j(c)\le1.
\]
Moreover, by \eqref{def-shat-jk},
\[
\widehat s_{j,k}^{(d)}(c)
=
\eta_j(c)
\widehat s_{j,k}^{(d)}(d+1)
\eta_k(c).
\]
Applying \eqref{odd-form-bound-critical} to
\[
\bigl(\eta_0(c)b_0,\ldots,\eta_N(c)b_N\bigr),
\]
we obtain
\[
0
\le
\sum_{j,k=0}^N
\widehat s_{j,k}^{(d)}(c)b_jb_k
\le
\frac1d
\sum_{j=0}^N
\eta_j(c)^2|b_j|^2
\le
\frac1d
\sum_{j=0}^N|b_j|^2,
\]
which proves \eqref{odd-form-bound}.
\end{proof}

\begin{corollary}\label{C5.7}
Let \(0<c\le d+1\), and let \(y\in H_0^1(0,\infty)\) have the
decomposition \eqref{y-even-odd-decomposition}. Then
\begin{equation}\label{odd-component-inequality}
Q_{d,c}(y_{\mathrm o})
\ge
cG_d(y_{\mathrm o}).
\end{equation}
Consequently,
\begin{equation}\label{Cd1-even-characterization}
\mathcal C_{d,1}\ge c
\end{equation}
if and only if
\begin{equation}\label{equiv-even-reduction}
Q_{d,c}(y_{\mathrm e})
\ge
cG_d(y_{\mathrm e})
\qquad
\text{for every }y\in H_0^1(0,\infty).
\end{equation}
\end{corollary}

\begin{proof}
By \eqref{Gd-odd-l2} and \eqref{odd-form-bound}, letting \(N\to\infty\),
we obtain
\[
cG_d(y_{\mathrm o})
\le
\frac1d
\sum_{j=0}^\infty |b_j|^2,
\]
where \(b=(b_j)_{j\in\mathbb N_0}\) is defined by
\eqref{def-bj-odd}. Hence, by \eqref{Qdc-odd-l2} and \(d\ge2\),
\[
cG_d(y_{\mathrm o})
\le
\frac1d Q_{d,c}(y_{\mathrm o})
\le
Q_{d,c}(y_{\mathrm o}),
\]
which proves \eqref{odd-component-inequality}.

By Lemma~\ref{L4.4}, \eqref{Cd1-even-characterization} is equivalent to
\[
Q_{d,c}(y)
\ge
cG_d(y)
\qquad
\text{for every }y\in H_0^1(0,\infty).
\]
Using \eqref{Qdc-even-odd}, \eqref{Gd-even-odd}, and
\eqref{odd-component-inequality}, this is equivalent to
\eqref{equiv-even-reduction}.
\end{proof}

\subsection{The even component}

In view of Corollary~\ref{C5.7}, it remains to analyze the even
component. This leads naturally to the operator
\[
S_c^{(d)}:\ell^2(\mathbb N_0)\to\ell^2(\mathbb N_0),
\]
defined by \eqref{eq-Sc}--\eqref{eq-sjk}.

We begin by establishing some basic properties of \(S_c^{(d)}\).

\begin{lemma}\label{L5.8}
Let \(0<c<d+3\). Then \(S_c^{(d)}\) is a positive Hilbert--Schmidt
operator on \(\ell^2(\mathbb N_0)\). In particular, \(S_c^{(d)}\) is
bounded and compact.
\end{lemma}

\begin{proof}
Throughout the proof, \(C>0\) denotes a constant whose value may change
from line to line and may depend only on \(d\) and \(c\).

We first establish the Hilbert--Schmidt property. By
\eqref{eq-hj-gammaj}, \eqref{eq-gammajd}, and Stirling's formula,
\[
h_j^{(d)}
\le
C(j+1)^{d-1},
\qquad
\gamma_j^{(d)}(c)
\le
C(j+1)^{-d-1},
\qquad
j\in\mathbb N_0.
\]
By \eqref{eq-sjk}, if \(j\le k\), then
\[
\begin{aligned}
\left|s_{j,k}^{(d)}(c)\right|^2
&=
\gamma_j^{(d)}(c)
\bigl(h_j^{(d)}\bigr)^2
\gamma_k^{(d)}(c)
\\
&\le
C
(j+1)^{d-3}
(k+1)^{-d-1}.
\end{aligned}
\]
Therefore, using the symmetry of \(s_{j,k}^{(d)}(c)\),
\[
\begin{aligned}
\sum_{j,k=0}^\infty
\left|s_{j,k}^{(d)}(c)\right|^2
&\le
C
\sum_{j=0}^\infty
(j+1)^{d-3}
\sum_{k=j}^\infty
(k+1)^{-d-1}
\\
&\le
C
\sum_{j=0}^\infty
(j+1)^{-3}
<
\infty.
\end{aligned}
\]
Thus \(S_c^{(d)}\) is Hilbert--Schmidt, and hence bounded and compact.

It remains to prove positivity. By \eqref{eq-hminus1-d} and
\eqref{eq-epsilon-d}, it follows that, for every
\(j,k\in\mathbb N_0\),
\begin{equation}\label{series-h}
h_{\min\{j,k\}}^{(d)}
=
\sum_{m=0}^{\min\{j,k\}}
\varepsilon_m^{(d)}.
\end{equation}
Moreover, by \eqref{exp-epsilon-md},
\[
\varepsilon_m^{(d)}>0,
\qquad
m\in\mathbb N_0.
\]

For \(N\in\mathbb N_0\) and \(x_0,\ldots,x_N\in\mathbb R\),
\eqref{eq-sjk} gives
\begin{equation}\label{pos-truncated}
\begin{aligned}
\sum_{j,k=0}^N
s_{j,k}^{(d)}(c)x_jx_k
&=
\sum_{m=0}^N
\varepsilon_m^{(d)}
\left(
\sum_{j=m}^N
\sqrt{\gamma_j^{(d)}(c)}\,x_j
\right)^2
\\
&\ge0.
\end{aligned}
\end{equation}
Now let \(x=(x_j)_{j\in\mathbb N_0}\in\ell^2(\mathbb N_0)\), and set
\[
x^{(N)}
=
(x_0,\ldots,x_N,0,\ldots).
\]
By \eqref{pos-truncated},
\[
\left\langle
S_c^{(d)}x^{(N)},x^{(N)}
\right\rangle_{\ell^2}
\ge0.
\]
Since \(S_c^{(d)}\) is bounded and
\[
x^{(N)}
\longrightarrow x
\qquad
\text{in }\ell^2(\mathbb N_0)
\text{ as }N\to\infty,
\]
we have
\[
\left\langle
S_c^{(d)}x^{(N)},x^{(N)}
\right\rangle_{\ell^2}
\longrightarrow
\left\langle
S_c^{(d)}x,x
\right\rangle_{\ell^2}
\qquad
\text{as }N\to\infty.
\]
Therefore,
\[
\left\langle S_c^{(d)}x,x\right\rangle_{\ell^2}
\ge0.
\]
Thus \(S_c^{(d)}\) is positive.
\end{proof}

We next study the dependence of \(S_c^{(d)}\) on the parameter \(c\).
 
\begin{lemma}\label{L5.9}
The map
\[
c\longmapsto
\|S_c^{(d)}\|_{\ell^2\to\ell^2}
\]
is continuous and strictly increasing on \((0,d+3)\). Moreover,
\[
\|S_c^{(d)}\|_{\ell^2\to\ell^2}
\longrightarrow0
\qquad
\text{as }c\to0^+,
\]
and
\[
\|S_c^{(d)}\|_{\ell^2\to\ell^2}
\longrightarrow\infty
\qquad
\text{as }c\to(d+3)^-.
\]
\end{lemma}

\begin{proof}
We first prove continuity. Fix \(c_0\in(0,d+3)\), and choose a compact
interval
\[
I\Subset(0,d+3)
\]
containing \(c_0\) in its interior.

Throughout this part of the proof, \(C>0\) denotes a generic constant,
possibly changing from line to line, independent of \(c\in I\) and
\(j,k\in\mathbb N_0\).

By \eqref{eq-hj-gammaj} and \eqref{eq-gammajd},
\[
h_j^{(d)}
\le
C(j+1)^{d-1},
\qquad
\gamma_j^{(d)}(c)
\le
C(j+1)^{-d-1},
\qquad
j\in\mathbb N_0.
\]
Together with \eqref{eq-sjk}, this gives, for \(j\le k\),
\[
\begin{aligned}
\left|s_{j,k}^{(d)}(c)\right|^2
&=
\gamma_j^{(d)}(c)
\bigl(h_j^{(d)}\bigr)^2
\gamma_k^{(d)}(c)
\\
&\le
C
(j+1)^{d-3}(k+1)^{-d-1}.
\end{aligned}
\]
Consequently,
\[
\left|
s_{j,k}^{(d)}(c)-s_{j,k}^{(d)}(c_0)
\right|^2
\le
C
(j+1)^{d-3}(k+1)^{-d-1},
\qquad
0\le j\le k.
\]
Moreover,
\[
\sum_{0\le j\le k<\infty}
(j+1)^{d-3}(k+1)^{-d-1}
<\infty.
\]
Indeed, if \(d=2\),
\[
\sum_{k=0}^\infty
(k+1)^{-3}
\sum_{j=0}^k(j+1)^{-1}
<\infty,
\]
whereas, if \(d\ge3\),
\[
\sum_{j=0}^k(j+1)^{d-3}
\le
(k+1)^{d-2},
\]
and therefore
\[
\sum_{k=0}^\infty
(k+1)^{-d-1}
\sum_{j=0}^k(j+1)^{d-3}
\le
\sum_{k=0}^\infty(k+1)^{-3}
<\infty.
\]

Moreover, by \eqref{eq-sjk}, for every \(j,k\in\mathbb N_0\), the map
\[
c\longmapsto s_{j,k}^{(d)}(c)
\]
is continuous on \((0,d+3)\).

By the dominated convergence theorem and the symmetry of the
matrix,
\[
\begin{aligned}
\left\|
S_c^{(d)}-S_{c_0}^{(d)}
\right\|_{\mathrm{HS}}^2
&=
\sum_{j,k=0}^\infty
\left|
s_{j,k}^{(d)}(c)-s_{j,k}^{(d)}(c_0)
\right|^2
\\
&\longrightarrow0
\qquad
\text{as }c\to c_0,
\end{aligned}
\]
which yields
\[
\|S_c^{(d)}\|_{\ell^2\to\ell^2}
\longrightarrow
\|S_{c_0}^{(d)}\|_{\ell^2\to\ell^2}
\qquad
\text{as }c\to c_0.
\]

We next prove strict monotonicity. Let
\[
0<c_1<c_2<d+3.
\]
By \eqref{eq-gammajd}, for every \(j\in\mathbb N_0\), the map
\[
c\longmapsto \gamma_j^{(d)}(c)
\]
is strictly increasing on \((0,d+3)\). Hence, by \eqref{eq-sjk},
\begin{equation}\label{sjk-strict-monotonicity}
s_{j,k}^{(d)}(c_1)
<
s_{j,k}^{(d)}(c_2),
\qquad
j,k\in\mathbb N_0.
\end{equation}
By Lemma~\ref{L5.8}, \(S_{c_1}^{(d)}\) is a positive
Hilbert--Schmidt operator on \(\ell^2(\mathbb N_0)\), and hence compact.
By the spectral theorem, there exists \(x\in\ell^2(\mathbb N_0)\), with
\[
\|x\|_{\ell^2}=1,
\]
such that
\[
S_{c_1}^{(d)}x
=
\|S_{c_1}^{(d)}\|_{\ell^2\to\ell^2}\,x.
\]
Set
\[
y=(|x_j|)_{j\in\mathbb N_0}.
\]
Then
\[
\|y\|_{\ell^2}=1.
\]
Since \(S_{c_1}^{(d)}\) is Hilbert--Schmidt and its entries are positive,
\[
\begin{aligned}
\|S_{c_1}^{(d)}\|_{\ell^2\to\ell^2}
&=
\left\langle S_{c_1}^{(d)}x,x\right\rangle_{\ell^2}
\\
&\le
\left\langle S_{c_1}^{(d)}y,y\right\rangle_{\ell^2}
\\
&\le
\|S_{c_1}^{(d)}\|_{\ell^2\to\ell^2}.
\end{aligned}
\]
Therefore,
\[
\left\langle S_{c_1}^{(d)}y,y\right\rangle_{\ell^2}
=
\|S_{c_1}^{(d)}\|_{\ell^2\to\ell^2}.
\]
Since \(y\neq0\), \eqref{sjk-strict-monotonicity} yields
\[
\left\langle S_{c_2}^{(d)}y,y\right\rangle_{\ell^2}
>
\left\langle S_{c_1}^{(d)}y,y\right\rangle_{\ell^2}.
\]
Consequently,
\[
\|S_{c_2}^{(d)}\|_{\ell^2\to\ell^2}
>
\|S_{c_1}^{(d)}\|_{\ell^2\to\ell^2}.
\]
Thus
\[
c\longmapsto
\|S_c^{(d)}\|_{\ell^2\to\ell^2}
\]
is strictly increasing on \((0,d+3)\).

Throughout the remainder of the proof, \(C>0\) denotes a generic
constant, possibly changing from line to line, independent of \(c\).

For \(0<c\le (d+3)/2\), by \eqref{eq-gammajd},
\[
\gamma_j^{(d)}(c)
\le
Cc(j+1)^{-d-1},
\qquad
j\in\mathbb N_0.
\]
Hence, using
\[
h_j^{(d)}
\le
C(j+1)^{d-1}\qquad
j\in\mathbb N_0,
\]
and \eqref{eq-sjk}, for \(0\le j\le k\),
\[
\begin{aligned}
\left|s_{j,k}^{(d)}(c)\right|^2
&=
\gamma_j^{(d)}(c)
\bigl(h_j^{(d)}\bigr)^2
\gamma_k^{(d)}(c)
\\
&\le
Cc^2
(j+1)^{d-3}(k+1)^{-d-1}.
\end{aligned}
\]
As in the proof of continuity,
\[
\sum_{0\le j\le k<\infty}
(j+1)^{d-3}(k+1)^{-d-1}
<\infty.
\]
Therefore, by symmetry,
\[
\|S_c^{(d)}\|_{\mathrm{HS}}^2
=
\sum_{j,k=0}^\infty
\left|s_{j,k}^{(d)}(c)\right|^2
\le
Cc^2.
\]
Consequently,
\[
\|S_c^{(d)}\|_{\ell^2\to\ell^2}
\le
\|S_c^{(d)}\|_{\mathrm{HS}}
\le
Cc
\longrightarrow0
\qquad
\text{as }c\to0^+.
\]

Finally,  since \(S_c^{(d)}\) is positive,
\[
\|S_c^{(d)}\|_{\ell^2\to\ell^2}
\ge
\left\langle S_c^{(d)}e_0,e_0\right\rangle_{\ell^2}
=
s_{0,0}^{(d)}(c).
\]
By \eqref{eq-sjk},
\[
s_{0,0}^{(d)}(c)
=
\gamma_0^{(d)}(c)h_0^{(d)}.
\]
Since \(h_0^{(d)}>0\) and, by \eqref{eq-gammajd},
\[
\gamma_0^{(d)}(c)
\longrightarrow\infty
\qquad
\text{as }c\to(d+3)^-,
\]
we obtain
\[
\|S_c^{(d)}\|_{\ell^2\to\ell^2}
\longrightarrow\infty
\qquad
\text{as }c\to(d+3)^-.
\]
This completes the proof. 
\end{proof}

We next establish the analogue of Lemma~\ref{L5.5} for the even
component of \(y\in H_0^1(0,\infty)\).

\begin{lemma}\label{L5.10}
Let \(0<c<d+3\), and let \(y\in H_0^1(0,\infty)\) have the expansion
\eqref{exp-y-psi-d}. Let
\[
y_{\mathrm e}
=
\sum_{j=0}^\infty
a_{2j}\psi_{2j}^{(d)}
\]
be its even component, and set
\begin{equation}\label{def-bj-even}
b_j
=
\sqrt{
(4j+d+3-c)\omega_{2j}^{(d)}
}\,a_{2j},
\qquad
j\in\mathbb N_0.
\end{equation}
Then \(b=(b_j)_{j\in\mathbb N_0}\in\ell^2(\mathbb N_0)\) and
\begin{equation}\label{Qdc-even-l2}
Q_{d,c}(y_{\mathrm e})
=
\|b\|_{\ell^2}^2.
\end{equation}
Moreover,
\begin{equation}\label{Gd-even-l2}
cG_d(y_{\mathrm e})
=
\left\langle
S_c^{(d)}b,b
\right\rangle_{\ell^2}.
\end{equation}
Conversely, for every \(b\in\ell^2(\mathbb N_0)\), there exists
\(y\in H_0^1(0,\infty)\) whose even component has coefficients
satisfying \eqref{def-bj-even}.
\end{lemma}

\begin{proof}
By \eqref{def-lambda-nd} and \eqref{Qdc-series},
\[
Q_{d,c}(y_{\mathrm e})
=
\sum_{j=0}^\infty
(4j+d+3-c)\omega_{2j}^{(d)}
|a_{2j}|^2.
\]
Since \(0<c<d+3\), all the coefficients in this series are positive.
Thus, by \eqref{def-bj-even},
\[
b\in\ell^2(\mathbb N_0),
\]
and \eqref{Qdc-even-l2} follows.

On the other hand, applying \eqref{Gd-series} to \(y_{\mathrm e}\) and
reindexing, we obtain
\[
cG_d(y_{\mathrm e})
=
\lim_{N\to\infty}
\sum_{j,k=0}^N
c\,g_{2j,2k}^{(d)}a_{2j}a_{2k}.
\]
Hence, by \eqref{eq-gammajd}, \eqref{eq-sjk}, \eqref{def-omega-nd},
\eqref{g-even-even}, and \eqref{def-bj-even},
\[
c\,g_{2j,2k}^{(d)}a_{2j}a_{2k}
=
s_{j,k}^{(d)}(c)b_jb_k,
\qquad
j,k\in\mathbb N_0.
\]
For
\[
b^{(N)}
=
(b_0,\ldots,b_N,0,\ldots),
\]
we therefore have
\[
cG_d(y_{\mathrm e})
=
\lim_{N\to\infty}
\left\langle
S_c^{(d)}b^{(N)},b^{(N)}
\right\rangle_{\ell^2}.
\]
Since
\[
b^{(N)}
\longrightarrow b
\qquad
\text{in }\ell^2(\mathbb N_0)
\text{ as }N\to\infty,
\]
and \(S_c^{(d)}\) is bounded by Lemma~\ref{L5.8}, we obtain
\eqref{Gd-even-l2}.

Conversely, let \(b=(b_j)_{j\in\mathbb N_0}\in\ell^2(\mathbb N_0)\), and set
\[
a_{2j}
=
\frac{
b_j
}{
\sqrt{
(4j+d+3-c)\omega_{2j}^{(d)}
}
},
\qquad
a_{2j+1}=0,
\qquad
j\in\mathbb N_0.
\]
Since \(0<c<d+3\),
\[
\frac{4j+d+3}{4j+d+3-c}
\le
\frac{d+3}{d+3-c},
\qquad
j\in\mathbb N_0.
\]
Hence
\[
\begin{aligned}
\sum_{j=0}^\infty
(4j+d+3)\omega_{2j}^{(d)}|a_{2j}|^2
&=
\sum_{j=0}^\infty
\frac{4j+d+3}{4j+d+3-c}|b_j|^2
\\
&\le
\frac{d+3}{d+3-c}
\|b\|_{\ell^2}^2
<\infty.
\end{aligned}
\]
For \(N\in\mathbb N_0\), set
\[
Y_N
=
\sum_{j=0}^N
a_{2j}\psi_{2j}^{(d)}.
\]
Then, for \(M>N\), by \eqref{def-lambda-nd} and \eqref{Ed-series},
\[
E_d(Y_M-Y_N)
=
\sum_{j=N+1}^M
(4j+d+3)\omega_{2j}^{(d)}
|a_{2j}|^2.
\]
Therefore,
\[
E_d(Y_M-Y_N)\longrightarrow0
\qquad
\text{as }M,N\to\infty.
\]
Since \(E_d^{1/2}\) is an equivalent norm on \(H_0^1(0,\infty)\),
there exists \(y\in H_0^1(0,\infty)\) such that
\[
Y_N\longrightarrow y
\qquad
\text{in }H_0^1(0,\infty)
\text{ as }N\to\infty.
\]
Thus
\[
y
=
\sum_{j=0}^\infty
a_{2j}\psi_{2j}^{(d)},
\]
its odd component vanishes, and its coefficients satisfy
\eqref{def-bj-even}. This completes the proof.
\end{proof}

\begin{corollary}\label{C5.11}
Let \(0<c\le d+1\). Then
\[
\mathcal C_{d,1}\ge c
\quad\Longleftrightarrow\quad
\|S_c^{(d)}\|_{\ell^2\to\ell^2}\le1.
\]
\end{corollary}

\begin{proof}
By Corollary~\ref{C5.7},
\[
\mathcal C_{d,1}\ge c
\]
if and only if
\[
Q_{d,c}(y_{\mathrm e})
\ge
cG_d(y_{\mathrm e})
\qquad
\text{for every }y\in H_0^1(0,\infty).
\]
By Lemma~\ref{L5.10}, this is equivalent to
\[
\|b\|_{\ell^2}^2
\ge
\left\langle S_c^{(d)}b,b\right\rangle_{\ell^2}
\qquad
\text{for every }b\in\ell^2(\mathbb N_0).
\]
Since \(S_c^{(d)}\) is positive by Lemma~\ref{L5.8}, the latter condition
is equivalent to
\[
\|S_c^{(d)}\|_{\ell^2\to\ell^2}\le1.
\]
This completes the proof.
\end{proof}

\subsection{Proof of Theorem~\ref{thm:spectral-characterization}}

\begin{proof}
By Lemma~\ref{L5.9}, the map
\[
c\longmapsto
\|S_c^{(d)}\|_{\ell^2\to\ell^2}
\]
is continuous and strictly increasing on \((0,d+3)\), with
\[
\|S_c^{(d)}\|_{\ell^2\to\ell^2}
\longrightarrow0
\qquad
\text{as }c\to0^+,
\]
and
\[
\|S_c^{(d)}\|_{\ell^2\to\ell^2}
\longrightarrow\infty
\qquad
\text{as }c\to(d+3)^-.
\]
Hence there exists a unique \(c_d^*\in(0,d+3)\) such that
\begin{equation}\label{def-c*}
\|S_{c_d^*}^{(d)}\|_{\ell^2\to\ell^2}=1.
\end{equation}

Let \(0<c\le d+1\). By Corollary~\ref{C5.11},
\[
\mathcal C_{d,1}\ge c
\quad\Longleftrightarrow\quad
\|S_c^{(d)}\|_{\ell^2\to\ell^2}\le1.
\]
Since
\[
c\longmapsto
\|S_c^{(d)}\|_{\ell^2\to\ell^2}
\]
is strictly increasing, it follows from \eqref{def-c*} that
\[
\|S_c^{(d)}\|_{\ell^2\to\ell^2}\le1
\quad\Longleftrightarrow\quad
c\le c_d^*.
\]
Thus
\[
\{c\in(0,d+1]:c\le\mathcal C_{d,1}\}
=
\{c\in(0,d+1]:c\le c_d^*\}.
\]
Taking the suprema of these two sets, we obtain
\[
\min\{d+1,\mathcal C_{d,1}\}
=
\min\{d+1,c_d^*\}.
\]
By \eqref{Cd-critical-sectors},
\[
\mathcal C_d
=
\min\{d+1,\mathcal C_{d,1}\},
\]
and hence
\[
\mathcal C_d
=
\min\{d+1,c_d^*\}.
\]
Finally, by \eqref{Cd-Hd-relation},
\[
\mathcal H_d
=
\frac{\mathcal C_d}{2}
=
\frac12\min\{d+1,c_d^*\},
\]
which proves \eqref{Hd-characterization}. This completes the proof.
\end{proof}

\subsection{Proof of Corollary~\ref{cor:Hd-dge4}}

\begin{proof}
Let \(d\ge4\). We first show that
\begin{equation}\label{claim-Cor}
\|S_{d+1}^{(d)}\|_{\ell^2\to\ell^2}<1.
\end{equation}

For \(j\in\mathbb N_0\), set
\begin{equation}\label{expression-pj}
p_j
=
\frac{2j+1}{4j+d+3}
\sqrt{\gamma_j^{(d)}(d+1)}.
\end{equation}
By \eqref{eq-gammajd}, for \(k\in\mathbb N_0\),
\[
\begin{aligned}
\sqrt{\gamma_k^{(d)}(d+1)}\,p_k
&=
\gamma_k^{(d)}(d+1)
\frac{2k+1}{4k+d+3}
\\
&=
\frac{2^{d+2}(d+1)}{(d-1)^2}
(4k+d+3)
\frac{\Gamma(2k+1)}{\Gamma(2k+d+3)}.
\end{aligned}
\]
Set
\[
r_k
=
\frac{\Gamma(2k+1)}{\Gamma(2k+d+1)},
\qquad
k\in\mathbb N_0.
\]
Using \eqref{eq-gamma-functional-relation}, we obtain
\[
r_k-r_{k+1}
=
d(4k+d+3)
\frac{\Gamma(2k+1)}{\Gamma(2k+d+3)}.
\]
Hence
\[
\sqrt{\gamma_k^{(d)}(d+1)}\,p_k
=
\frac{2^{d+2}(d+1)}{d(d-1)^2}
(r_k-r_{k+1}).
\]
Since, by Stirling's formula,
\[
r_k\longrightarrow0
\qquad
\text{as }k\to\infty,
\]
it follows that
\[
\begin{aligned}
\sum_{k=m}^\infty
\sqrt{\gamma_k^{(d)}(d+1)}\,p_k
&=
\frac{2^{d+2}(d+1)}
{d(d-1)^2}
r_m
\\
&=
\frac{2^{d+2}(d+1)}
{d(d-1)^2}
\frac{\Gamma(2m+1)}{\Gamma(2m+d+1)}.
\end{aligned}
\]
Together with \eqref{exp-epsilon-md} and \eqref{series-h}, this gives,
for every \(j\in\mathbb N_0\),
\[
\begin{aligned}
\sum_{k=0}^\infty
h_{\min\{j,k\}}^{(d)}
\sqrt{\gamma_k^{(d)}(d+1)}\,p_k
&=
\sum_{m=0}^j
\varepsilon_m^{(d)}
\sum_{k=m}^\infty
\sqrt{\gamma_k^{(d)}(d+1)}\,p_k
\\
&=
\frac{2(d+1)}{d}
\sum_{m=0}^j
\frac{1}
{(4m+d-1)(4m+d+3)}
\\
&=
\frac{d+1}{2d}
\left(
\frac{1}{d-1}
-
\frac{1}{4j+d+3}
\right)
\\
&=
\frac{
2(d+1)(j+1)
}{
d(d-1)(4j+d+3)
}.
\end{aligned}
\]
Therefore, by \eqref{eq-sjk} and \eqref{expression-pj}, for every
\(j\in\mathbb N_0\),
\[
\begin{aligned}
\sum_{k=0}^\infty
s_{j,k}^{(d)}(d+1)p_k
&=
\sqrt{\gamma_j^{(d)}(d+1)}
\sum_{k=0}^\infty
h_{\min\{j,k\}}^{(d)}
\sqrt{\gamma_k^{(d)}(d+1)}\,p_k
\\
&=
\frac{
2(d+1)(j+1)
}{
d(d-1)(4j+d+3)
}
\sqrt{\gamma_j^{(d)}(d+1)}
\\
&=
\frac{
2(d+1)(j+1)
}{
d(d-1)(2j+1)
}
p_j.
\end{aligned}
\]
Consequently,
\begin{equation}\label{weighted-row-estimate}
\sum_{k=0}^\infty
s_{j,k}^{(d)}(d+1)p_k
\le
\frac{2(d+1)}{d(d-1)}p_j,
\qquad
j\in\mathbb N_0.
\end{equation}

For \(x\in\ell^2(\mathbb N_0)\), by
\eqref{weighted-row-estimate} and the Cauchy--Schwarz inequality, for
every \(j\in\mathbb N_0\),
\[
\begin{aligned}
\left|
\sum_{k=0}^\infty
s_{j,k}^{(d)}(d+1)x_k
\right|^2
&\le
\left(
\sum_{k=0}^\infty
s_{j,k}^{(d)}(d+1)p_k
\right)
\left(
\sum_{k=0}^\infty
s_{j,k}^{(d)}(d+1)
\frac{|x_k|^2}{p_k}
\right)
\\
&\le
\frac{2(d+1)}{d(d-1)}p_j
\sum_{k=0}^\infty
s_{j,k}^{(d)}(d+1)
\frac{|x_k|^2}{p_k}.
\end{aligned}
\]
Summing over \(j\in\mathbb N_0\), we obtain
\[
\begin{aligned}
\|S_{d+1}^{(d)}x\|_{\ell^2}^2
&=
\sum_{j=0}^\infty
\left|
\sum_{k=0}^\infty
s_{j,k}^{(d)}(d+1)x_k
\right|^2
\\
&\le
\frac{2(d+1)}{d(d-1)}
\sum_{j=0}^\infty
p_j
\sum_{k=0}^\infty
s_{j,k}^{(d)}(d+1)
\frac{|x_k|^2}{p_k}.
\end{aligned}
\]
By Tonelli's theorem and the symmetry of \(s_{j,k}^{(d)}(d+1)\),
\[
\|S_{d+1}^{(d)}x\|_{\ell^2}^2
\le
\frac{2(d+1)}{d(d-1)}
\sum_{k=0}^\infty
\frac{|x_k|^2}{p_k}
\left(
\sum_{j=0}^\infty
s_{k,j}^{(d)}(d+1)p_j
\right).
\]
Applying \eqref{weighted-row-estimate} again, we obtain
\[
\begin{aligned}
\|S_{d+1}^{(d)}x\|_{\ell^2}^2
&\le
\left(
\frac{2(d+1)}{d(d-1)}
\right)^2
\sum_{k=0}^\infty |x_k|^2
\\
&=
\left(
\frac{2(d+1)}{d(d-1)}
\right)^2
\|x\|_{\ell^2}^2.
\end{aligned}
\]
Hence
\[
\|S_{d+1}^{(d)}\|_{\ell^2\to\ell^2}
\le
\frac{2(d+1)}{d(d-1)}.
\]
Since \(d\ge4\),
\[
\frac{2(d+1)}{d(d-1)}<1,
\]
which proves \eqref{claim-Cor}.

Finally, by \eqref{claim-Cor}, Lemma~\ref{L5.9}, and
\[
\|S_{c_d^*}^{(d)}\|_{\ell^2\to\ell^2}=1,
\]
we obtain
\[
c_d^*>d+1.
\]
Therefore, by \eqref{Hd-characterization},
\[
\mathcal H_d
=
\frac12\min\{d+1,c_d^*\}
=
\frac{d+1}{2}.
\]
This completes the proof.
\end{proof}

\subsection{Proof of Corollary~\ref{cor:Hd-d23}}

\begin{proof}
Let \(d\in\{2,3\}\). Since \(S_{d+1}^{(d)}\) is positive, we have
\[
\|S_{d+1}^{(d)}\|_{\ell^2\to\ell^2}
\ge
\left\langle S_{d+1}^{(d)}e_0,e_0\right\rangle_{\ell^2}
=
s_{0,0}^{(d)}(d+1).
\]
By \eqref{eq-hj-gammaj}, \eqref{eq-gammajd}, and \eqref{eq-sjk},
\[
s_{0,0}^{(d)}(d+1)
=
\gamma_0^{(d)}(d+1)h_0^{(d)}
=
\frac{2(d+3)}{(d-1)(d+2)}
>1.
\]
Therefore,
\[
\|S_{d+1}^{(d)}\|_{\ell^2\to\ell^2}>1.
\]
Since
\[
\|S_{c_d^*}^{(d)}\|_{\ell^2\to\ell^2}=1
\]
and the map
\[
c\longmapsto
\|S_c^{(d)}\|_{\ell^2\to\ell^2}
\]
is strictly increasing by Lemma~\ref{L5.9}, it follows that
\[
c_d^*<d+1.
\]
Finally, \eqref{Hd-characterization} gives
\[
\mathcal H_d
=
\frac12\min\{d+1,c_d^*\}
=
\frac{c_d^*}{2},
\]
which proves \eqref{Hd-d-2-3}. This completes the proof.
\end{proof}

\section{Quantitative approximation of the spectral parameters}\label{sec6}

In this section, we prove Theorem~\ref{thm:quantitative-estimate} and
derive its consequences. We first develop the discrete recursive scheme
used to approximate the critical spectral parameters \(c_2^*\) and
\(c_3^*\), and establish the auxiliary estimates needed for the proof.
We then obtain effective two-sided bounds for the sharp constants
\(\mathcal H_2\) and \(\mathcal H_3\), including the explicit numerical
estimates stated in Corollary~\ref{cor:numerical-Hd}.

\subsection{The discrete recursive scheme}

We begin with an identity that relates the recursive sequences
\(\mathbb A_j^{(d)}(c)\) and \(\mathbb F_j^{(d)}(c)\), defined by
\eqref{eq-shooting-initial-d}--\eqref{eq-shooting-A-d}, to the kernel of
\(S_c^{(d)}\).

\begin{lemma}\label{L6.1}
Let \(d\in\{2,3\}\), \(0<c<d+1\), and \(N\in\mathbb N_0\). Then
\begin{equation}\label{eq:recursive-identity}
\sum_{k=0}^{N}
h_{\min\{j,k\}}^{(d)}
\gamma_k^{(d)}(c)
\mathbb A_k^{(d)}(c)
=
\mathbb A_j^{(d)}(c)
-
h_j^{(d)}
\mathbb F_{N+1}^{(d)}(c),
\qquad
0\le j\le N.
\end{equation}
\end{lemma}

\begin{proof}
For \(0\le j\le N\), set
\[
T_j
=
\sum_{k=0}^{N}
h_{\min\{j,k\}}^{(d)}
\gamma_k^{(d)}(c)\mathbb A_k^{(d)}(c).
\]

Suppose first that \(N=0\). By
\eqref{eq-shooting-F-d},
\[
\mathbb F_1^{(d)}(c)
=
\mathbb F_0^{(d)}(c)
-
\gamma_0^{(d)}(c)\mathbb A_0^{(d)}(c).
\]
Together with \eqref{eq-shooting-initial-d}, this yields
\[
\mathbb A_0^{(d)}(c)
-
h_0^{(d)}\mathbb F_1^{(d)}(c)
=
h_0^{(d)}\gamma_0^{(d)}(c)
\mathbb A_0^{(d)}(c)
=
T_0,
\]
which proves \eqref{eq:recursive-identity} when \(N=0\).

Assume now that \(N\ge1\). By \eqref{eq-shooting-F-d}, for
\(0\le j\le N\),
\begin{equation}\label{eq:F-finite-tail}
\sum_{k=j}^{N}
\gamma_k^{(d)}(c)\mathbb A_k^{(d)}(c)
=
\mathbb F_j^{(d)}(c)
-
\mathbb F_{N+1}^{(d)}(c).
\end{equation}
Using \eqref{eq-shooting-initial-d} and
\eqref{eq:F-finite-tail} with \(j=0\), we obtain
\begin{equation}\label{eq-sum-h-gamma-A}
\begin{aligned}
T_0
&=
h_0^{(d)}
\sum_{k=0}^{N}
\gamma_k^{(d)}(c)\mathbb A_k^{(d)}(c)
\\
&=
h_0^{(d)}
\left(
\mathbb F_0^{(d)}(c)
-
\mathbb F_{N+1}^{(d)}(c)
\right)
\\
&=
\mathbb A_0^{(d)}(c)
-
h_0^{(d)}
\mathbb F_{N+1}^{(d)}(c).
\end{aligned}
\end{equation}
Moreover, for \(1\le j\le N\), by \eqref{eq-epsilon-d} and
\eqref{eq:F-finite-tail},
\[
\begin{aligned}
T_j-T_{j-1}
&=
\varepsilon_j^{(d)}
\sum_{k=j}^{N}
\gamma_k^{(d)}(c)\mathbb A_k^{(d)}(c)
\\
&=
\varepsilon_j^{(d)}
\left(
\mathbb F_j^{(d)}(c)
-
\mathbb F_{N+1}^{(d)}(c)
\right).
\end{aligned}
\]
On the other hand, by \eqref{eq-epsilon-d} and
\eqref{eq-shooting-A-d},
\[
\varepsilon_j^{(d)}\mathbb F_j^{(d)}(c)
=
\mathbb A_j^{(d)}(c)
-
\mathbb A_{j-1}^{(d)}(c),
\qquad
\varepsilon_j^{(d)}
=
h_j^{(d)}-h_{j-1}^{(d)}.
\]
Hence
\[
T_j-T_{j-1}
=
\mathbb A_j^{(d)}(c)
-
\mathbb A_{j-1}^{(d)}(c)
-
\left(
h_j^{(d)}-h_{j-1}^{(d)}
\right)
\mathbb F_{N+1}^{(d)}(c).
\]
Therefore,
\[
T_j-
\left(
\mathbb A_j^{(d)}(c)
-
h_j^{(d)}\mathbb F_{N+1}^{(d)}(c)
\right)
=
T_{j-1}-
\left(
\mathbb A_{j-1}^{(d)}(c)
-
h_{j-1}^{(d)}\mathbb F_{N+1}^{(d)}(c)
\right).
\]
By \eqref{eq-sum-h-gamma-A}, the expression on the right-hand side is
zero when \(j=1\). It follows by induction that
\[
T_j
=
\mathbb A_j^{(d)}(c)
-
h_j^{(d)}\mathbb F_{N+1}^{(d)}(c),
\qquad
0\le j\le N,
\]
which is \eqref{eq:recursive-identity}.
\end{proof}

We next use Lemma~\ref{L6.1} to characterize the spectral threshold
\(c_d^*\) in terms of the sign of the recursive sequence
\(\mathbb F_j^{(d)}(c)\).

\begin{lemma}\label{L6.2}
Let \(d\in\{2,3\}\) and \(0<c<d+1\). The following statements hold:
\begin{itemize}
	\item[\rm(i)] For every \(j\in\mathbb N_0\),
	\begin{equation}\label{eq:A-F-representation}
	\mathbb A_j^{(d)}(c)
	=
	\sum_{k=0}^{j}
	\varepsilon_k^{(d)}
	\mathbb F_k^{(d)}(c).
	\end{equation}
	In particular, if
	\[
	\mathbb F_k^{(d)}(c)>0,
	\qquad
	0\le k\le j,
	\]
	then
	\[
	\mathbb A_j^{(d)}(c)>0.
	\]

	\item[\rm(ii)] We have
	\[
	c\le c_d^*
	\]
	if and only if
	\[
	\mathbb F_j^{(d)}(c)>0,
	\qquad
	j\in\mathbb N_0.
	\]

	\item[\rm(iii)] At the critical value,
	\begin{equation}\label{eq:F-critical-limit}
	\mathbb F_j^{(d)}(c_d^*)\downarrow0
	\qquad
	\text{as }j\to\infty.
	\end{equation}
\end{itemize}
\end{lemma}

\begin{proof}
(i) By \eqref{eq-hminus1-d}, \eqref{eq-epsilon-d}, and
\eqref{eq-shooting-initial-d},
\[
\varepsilon_0^{(d)}
=
h_0^{(d)},
\qquad
\mathbb A_0^{(d)}(c)
=
1,
\qquad
\mathbb F_0^{(d)}(c)
=
\frac{1}{h_0^{(d)}},
\]
and hence
\[
\mathbb A_0^{(d)}(c)
=
\varepsilon_0^{(d)}\mathbb F_0^{(d)}(c).
\]
This shows that \eqref{eq:A-F-representation} holds for \(j=0\).
Using \eqref{eq-shooting-A-d}, it follows by induction that
\eqref{eq:A-F-representation} holds for every \(j\in\mathbb N_0\).

Moreover, by \eqref{exp-epsilon-md},
\[
\varepsilon_k^{(d)}>0,
\qquad
k\in\mathbb N_0.
\]
Hence, if
\[
\mathbb F_k^{(d)}(c)>0,
\qquad
0\le k\le j,
\]
then \eqref{eq:A-F-representation} yields
\[
\mathbb A_j^{(d)}(c)>0.
\]
This proves (i).

\medskip

\noindent (ii) Assume first that
\[
0<c\le c_d^*.
\]
By Theorem~\ref{thm:spectral-characterization} and Lemma~\ref{L5.9},
\begin{equation}\label{eq-conseq-norm-Sc}
\|S_c^{(d)}\|_{\ell^2\to\ell^2}
\le
\|S_{c_d^*}^{(d)}\|_{\ell^2\to\ell^2}
=
1.
\end{equation}

Suppose, by contradiction, that
\(\mathbb F_j^{(d)}(c)\) is not positive for every
\(j\in\mathbb N_0\). Since
\[
\mathbb F_0^{(d)}(c)=\frac1{h_0^{(d)}}>0,
\]
there exists \(N\in\mathbb N_0\) such that \(N+1\) is the first index
for which
\(\mathbb F_{N+1}^{(d)}(c)\le0\). Hence
\begin{equation}\label{eq2-L6.2}
\mathbb F_j^{(d)}(c)>0,
\qquad
0\le j\le N,
\end{equation}
whereas
\begin{equation}\label{sign-F-N+1-c}
\mathbb F_{N+1}^{(d)}(c)\le0.
\end{equation}
By \eqref{eq2-L6.2} and part (i),
\begin{equation}\label{sign-Aj-c}
\mathbb A_j^{(d)}(c)>0,
\qquad
0\le j\le N.
\end{equation}

Define \(p^{(N)}=(p_j^{(N)})_{j\in\mathbb N_0}\) by
\[
p_j^{(N)}
=
\begin{cases}
\sqrt{\gamma_j^{(d)}(c)}\,
\mathbb A_j^{(d)}(c),
& 0\le j\le N,\\
0,
& j>N.
\end{cases}
\]
By \eqref{eq:recursive-identity} and \eqref{sign-F-N+1-c}, for
\(0\le j\le N\),
\begin{align}
\bigl(S_c^{(d)}p^{(N)}\bigr)_j
&=
\sqrt{\gamma_j^{(d)}(c)}
\sum_{k=0}^{N}
h_{\min\{j,k\}}^{(d)}
\gamma_k^{(d)}(c)
\mathbb A_k^{(d)}(c)
\nonumber\\
&=
\sqrt{\gamma_j^{(d)}(c)}
\left(
\mathbb A_j^{(d)}(c)
-
h_j^{(d)}
\mathbb F_{N+1}^{(d)}(c)
\right)
\nonumber\\
&\ge
p_j^{(N)}.
\label{key-ineq}
\end{align}

If
\[
\mathbb F_{N+1}^{(d)}(c)<0,
\]
then the inequality in \eqref{key-ineq} is strict for every
\(0\le j\le N\), and hence
\[
\|S_c^{(d)}p^{(N)}\|_{\ell^2}
>
\|p^{(N)}\|_{\ell^2}.
\]

If
\[
\mathbb F_{N+1}^{(d)}(c)=0,
\]
then
\[
\bigl(S_c^{(d)}p^{(N)}\bigr)_j
=
p_j^{(N)},
\qquad
0\le j\le N,
\]
whereas, for \(j>N\), by \eqref{sign-Aj-c},
\[
\bigl(S_c^{(d)}p^{(N)}\bigr)_j
=
\sqrt{\gamma_j^{(d)}(c)}
\sum_{k=0}^{N}
h_k^{(d)}
\gamma_k^{(d)}(c)
\mathbb A_k^{(d)}(c)
>
0
=
p_j^{(N)}.
\]
Thus again,
\[
\|S_c^{(d)}p^{(N)}\|_{\ell^2}
>
\|p^{(N)}\|_{\ell^2}.
\]

Hence, in both cases,
\[
\|S_c^{(d)}\|_{\ell^2\to\ell^2}>1,
\]
which contradicts \eqref{eq-conseq-norm-Sc}. Therefore,
\[
\mathbb F_j^{(d)}(c)>0,
\qquad
j\in\mathbb N_0.
\]

Conversely, assume that
\[
\mathbb F_j^{(d)}(c)>0,
\qquad
j\in\mathbb N_0.
\]
By part (i),
\begin{equation}\label{pos-Aj-every-j}
\mathbb A_j^{(d)}(c)>0,
\qquad
j\in\mathbb N_0.
\end{equation}
Suppose, by contradiction, that
\[
c>c_d^*.
\]
By \eqref{eq-shooting-F-d} and \eqref{pos-Aj-every-j}, the sequence
\(\bigl(\mathbb F_j^{(d)}(c)\bigr)_{j\in\mathbb N_0}\) is strictly
decreasing. Hence there exists \(L\ge0\) such that
\[
\mathbb F_j^{(d)}(c)\longrightarrow L
\qquad\text{as }j\to\infty.
\]
Moreover,
\[
\sum_{k=0}^{N}
\gamma_k^{(d)}(c)\mathbb A_k^{(d)}(c)
=
\mathbb F_0^{(d)}(c)
-
\mathbb F_{N+1}^{(d)}(c),
\]
and therefore
\[
\sum_{k=0}^{\infty}
\gamma_k^{(d)}(c)\mathbb A_k^{(d)}(c)
=
\mathbb F_0^{(d)}(c)-L
<\infty.
\]
Fix \(j\in\mathbb N_0\). Letting \(N\to\infty\) in
\eqref{eq:recursive-identity}, we obtain
\begin{equation}\label{eq1-converse-L6.2}
\sum_{k=0}^{\infty}
h_{\min\{j,k\}}^{(d)}
\gamma_k^{(d)}(c)\mathbb A_k^{(d)}(c)
=
\mathbb A_j^{(d)}(c)-h_j^{(d)}L
\le
\mathbb A_j^{(d)}(c).
\end{equation}

Set
\[
q_j
=
\sqrt{\gamma_j^{(d)}(c)}
\mathbb A_j^{(d)}(c),
\qquad
j\in\mathbb N_0.
\]
By \eqref{eq1-converse-L6.2}, for every \(j\in\mathbb N_0\),
\begin{equation}\label{est-sum-sjk-qk}
\begin{aligned}
\sum_{k=0}^{\infty}
s_{j,k}^{(d)}(c)q_k
&=
\sqrt{\gamma_j^{(d)}(c)}
\sum_{k=0}^{\infty}
h_{\min\{j,k\}}^{(d)}
\gamma_k^{(d)}(c)
\mathbb A_k^{(d)}(c)
\\
&\le q_j.
\end{aligned}
\end{equation}

Let \(x=(x_j)_{j\in\mathbb N_0}\in\ell^2(\mathbb N_0)\). Using the symmetry of \(S_c^{(d)}\), Tonelli's theorem, Young's inequality, and
\eqref{est-sum-sjk-qk}, we obtain 
\[
\begin{aligned}
\sum_{j,k=0}^{\infty}
s_{j,k}^{(d)}(c)|x_jx_k|
&\le
\sum_{j=0}^{\infty}
\frac{|x_j|^2}{q_j}
\sum_{k=0}^{\infty}
s_{j,k}^{(d)}(c)q_k
\\
&\le
\|x\|_{\ell^2}^2.
\end{aligned}
\]
Consequently,
\[
\left\langle S_c^{(d)}x,x\right\rangle_{\ell^2}
\le
\|x\|_{\ell^2}^2.
\]
Since \(S_c^{(d)}\) is positive and self-adjoint,
\[
\|S_c^{(d)}\|_{\ell^2\to\ell^2}\le1.
\]
This contradicts \(c>c_d^*\), in view of
Theorem~\ref{thm:spectral-characterization} and Lemma~\ref{L5.9}.
Consequently,
\[
c\le c_d^*.
\]
This completes the proof of part (ii).

\medskip

\noindent (iii) By Theorem~\ref{thm:spectral-characterization} and
Lemma~\ref{L5.9},
\[
\|S_{c_d^*}^{(d)}\|_{\ell^2\to\ell^2}=1,
\]
and \(S_{c_d^*}^{(d)}\) is positive and compact.
Hence, by the spectral theorem, there exists a nonzero
\(\tau=(\tau_j)_{j\in\mathbb N_0}\in\ell^2(\mathbb N_0)\) such that
\begin{equation}\label{eigenvalue-eq}
S_{c_d^*}^{(d)}\tau=\tau.
\end{equation}
Since the entries of \(S_{c_d^*}^{(d)}\) are nonnegative,
\[
\left\langle
S_{c_d^*}^{(d)}|\tau|,|\tau|
\right\rangle_{\ell^2}
\ge
\left\langle
S_{c_d^*}^{(d)}\tau,\tau
\right\rangle_{\ell^2}
=
\|\tau\|_{\ell^2}^2,
\]
where
\[
|\tau|=(|\tau_j|)_{j\in\mathbb N_0}.
\]
On the other hand,
\[
\left\langle
S_{c_d^*}^{(d)}|\tau|,|\tau|
\right\rangle_{\ell^2}
\le
\|S_{c_d^*}^{(d)}\|_{\ell^2\to\ell^2}
\|\tau\|_{\ell^2}^2
=
\|\tau\|_{\ell^2}^2.
\]
Hence
\[
\left\langle
\left(I-S_{c_d^*}^{(d)}\right)|\tau|,|\tau|
\right\rangle_{\ell^2}
=0.
\]
Since \(I-S_{c_d^*}^{(d)}\) is positive, it follows that
\[
S_{c_d^*}^{(d)}|\tau|=|\tau|.
\]
Replacing \(\tau\) by \(|\tau|\), we may therefore assume that
\[
\tau_j\ge0,
\qquad
j\in\mathbb N_0.
\]
Since \(\tau\neq0\), there exists \(k_0\in\mathbb N_0\) such that
\(\tau_{k_0}>0\). Hence, by \eqref{eigenvalue-eq} and the strict
positivity of the entries of \(S_{c_d^*}^{(d)}\),
\[
\tau_j
=
\sum_{k=0}^{\infty}
s_{j,k}^{(d)}(c_d^*)\tau_k
\ge
s_{j,k_0}^{(d)}(c_d^*)\tau_{k_0}
>0,
\qquad
j\in\mathbb N_0.
\]
After multiplying \(\tau\) by a positive constant, we may assume that
\begin{equation}\label{eq-normal-tau0}
\frac{\tau_0}{\sqrt{\gamma_0^{(d)}(c_d^*)}}=1.
\end{equation}

Set
\[
\widetilde{\mathbb A}_j
=
\frac{\tau_j}{\sqrt{\gamma_j^{(d)}(c_d^*)}},
\qquad
j\in\mathbb N_0.
\]
By \eqref{eq-normal-tau0},
\begin{equation}\label{eq-Atilde0}
\widetilde{\mathbb A}_0=1.
\end{equation}
Dividing \eqref{eigenvalue-eq} componentwise by
\(\sqrt{\gamma_j^{(d)}(c_d^*)}\), we obtain
\begin{equation}\label{eq-eigen-Atilde}
\widetilde{\mathbb A}_j
=
\sum_{k=0}^{\infty}
h_{\min\{j,k\}}^{(d)}
\gamma_k^{(d)}(c_d^*)
\widetilde{\mathbb A}_k,
\qquad
j\in\mathbb N_0.
\end{equation}
Taking \(j=0\) in \eqref{eq-eigen-Atilde} and using
\eqref{eq-Atilde0}, we obtain
\[
1
=
h_0^{(d)}
\sum_{k=0}^{\infty}
\gamma_k^{(d)}(c_d^*)
\widetilde{\mathbb A}_k.
\]
Hence
\begin{equation}\label{eq-Ftilde0}
\sum_{k=0}^{\infty}
\gamma_k^{(d)}(c_d^*)
\widetilde{\mathbb A}_k
=
\frac{1}{h_0^{(d)}}.
\end{equation}

For \(j\in\mathbb N_0\), set
\begin{equation}\label{eq-Ftilde-def}
\widetilde{\mathbb F}_j
=
\sum_{k=j}^{\infty}
\gamma_k^{(d)}(c_d^*)
\widetilde{\mathbb A}_k.
\end{equation}
Then \eqref{eq-Ftilde0} gives
\[
\widetilde{\mathbb F}_0
=
\frac{1}{h_0^{(d)}},
\]
and, by \eqref{eq-Ftilde-def},
\begin{equation}\label{eq-Ftilde-recursion}
\widetilde{\mathbb F}_{j+1}
=
\widetilde{\mathbb F}_j
-
\gamma_j^{(d)}(c_d^*)
\widetilde{\mathbb A}_j,
\qquad
j\in\mathbb N_0.
\end{equation}

Moreover, subtracting \eqref{eq-eigen-Atilde} at the indices \(j+1\)
and \(j\), and using \eqref{eq-epsilon-d}, we obtain
\begin{align}
\widetilde{\mathbb A}_{j+1}
-
\widetilde{\mathbb A}_j
&=
\left(
h_{j+1}^{(d)}-h_j^{(d)}
\right)
\sum_{k=j+1}^{\infty}
\gamma_k^{(d)}(c_d^*)
\widetilde{\mathbb A}_k
\nonumber\\
&=
\varepsilon_{j+1}^{(d)}
\widetilde{\mathbb F}_{j+1},
\qquad
j\in\mathbb N_0.
\label{eq-Atilde-recursion}
\end{align}
Thus, by \eqref{eq-Atilde0}, \eqref{eq-Ftilde0},
\eqref{eq-Ftilde-recursion}, and \eqref{eq-Atilde-recursion}, the
sequences
\[
\bigl(\widetilde{\mathbb A}_j\bigr)_{j\in\mathbb N_0},
\qquad
\bigl(\widetilde{\mathbb F}_j\bigr)_{j\in\mathbb N_0},
\]
satisfy the same initial conditions and recursive relations as
\[
\bigl(\mathbb A_j^{(d)}(c_d^*)\bigr)_{j\in\mathbb N_0},
\qquad
\bigl(\mathbb F_j^{(d)}(c_d^*)\bigr)_{j\in\mathbb N_0}.
\]
Therefore,
\begin{equation}\label{eq-identification-tilde-AF}
\widetilde{\mathbb A}_j
=
\mathbb A_j^{(d)}(c_d^*),
\qquad
\widetilde{\mathbb F}_j
=
\mathbb F_j^{(d)}(c_d^*),
\qquad
j\in\mathbb N_0.
\end{equation}

Since all the terms in \eqref{eq-Ftilde-def} are strictly positive and
the series in \eqref{eq-Ftilde0} converges, its tails satisfy
\[
\widetilde{\mathbb F}_j\downarrow0
\qquad\text{as }j\to\infty.
\]
Together with \eqref{eq-identification-tilde-AF}, this yields \eqref{eq:F-critical-limit}.
This proves (iii).
\end{proof}

For \(m\in\mathbb N_0\), define
\begin{equation}\label{eq-alpha-2m}
\alpha_{2,m}(c)
=
\frac{
20m+45-(2m+7)c
}{
2(8m+18-c)
},
\qquad
0<c<3,
\end{equation}
and
\begin{equation}\label{eq-alpha-3m}
\alpha_{3,m}(c)
=
\frac{
36m+90-(2m+11)c
}{
4(6m+15-c)
},
\qquad
0<c<4.
\end{equation}

For \(j\in\mathbb N_0\) with \(j\ge m\), set
\begin{equation}\label{eq-tail-aR-d2}
a_j^{(2,m)}(c)
=
\frac{j+\alpha_{2,m}(c)}{4j+5},
\qquad
R_j^{(2,m)}(c)
=
\frac{
4\bigl(5-4\alpha_{2,m}(c)\bigr)
}{
(j+1)(2j+1)
},
\qquad
0<c<3,
\end{equation}
and
\begin{equation}\label{eq-tail-aR-d3}
a_j^{(3,m)}(c)
=
\frac{j+\alpha_{3,m}(c)}{2j+3},
\qquad
R_j^{(3,m)}(c)
=
\frac{
8\bigl(3-2\alpha_{3,m}(c)\bigr)
}{
(j+1)(2j+1)(2j+3)
},
\qquad
0<c<4.
\end{equation}
Using \eqref{eq-frakb-d}, a direct computation gives
\begin{equation}\label{eq-frakb-R-a}
\mathfrak b_m^{(d)}(c)
=
\frac{
R_{m+1}^{(d,m)}(c)
}{
a_m^{(d,m)}(c)
},
\qquad
m\in\mathbb N_0,\quad
d\in\{2,3\},\quad
0<c<d+1.
\end{equation}

\begin{lemma}\label{L6.3}
Let \(d\in\{2,3\}\), \(m\in\mathbb N_0\), and \(0<c<d+1\). Assume that
\begin{equation}\label{criterion-Aj}
\mathbb A_j^{(d)}(c)>0,
\qquad
0\le j\le m,
\end{equation}
and
\begin{equation}\label{eq-finite-criterion}
\mathbb F_{m+1}^{(d)}(c)
>
\mathfrak b_m^{(d)}(c)
\mathbb A_m^{(d)}(c).
\end{equation}
Then
\[
c<c_d^*.
\]
\end{lemma}

\begin{proof}
If \(d=2\), then by \eqref{eq-alpha-2m} and \(0<c<3\),
\[
\alpha_{2,m}(c)>0,
\]
while, if \(d=3\), then by \eqref{eq-alpha-3m} and \(0<c<4\),
\[
\alpha_{3,m}(c)>0.
\]
Consequently, by \eqref{eq-tail-aR-d2} and \eqref{eq-tail-aR-d3},
\begin{equation}\label{eq-tail-aR-positive}
a_j^{(d,m)}(c)>0,
\qquad
j\ge m.
\end{equation}

A direct computation using \eqref{eq-epsilon-d-explicit},
\eqref{eq-tail-aR-d2}, and \eqref{eq-tail-aR-d3} gives
\begin{equation}\label{eq-tail-increment}
a_j^{(d,m)}(c)-a_{j-1}^{(d,m)}(c)
=
\varepsilon_j^{(d)}R_j^{(d,m)}(c),
\qquad
j\ge m+1.
\end{equation}
Moreover, using \eqref{eq-gammajd}, \eqref{eq-alpha-2m},
\eqref{eq-alpha-3m}, \eqref{eq-tail-aR-d2}, and
\eqref{eq-tail-aR-d3}, we obtain
\begin{equation}\label{eq-tail-majorization}
\gamma_j^{(d)}(c)a_j^{(d,m)}(c)
\le
R_j^{(d,m)}(c)-R_{j+1}^{(d,m)}(c),
\qquad
j\ge m+1.
\end{equation}

For \(j\ge m+1\), set
\[
W_j
=
\sum_{k=j}^{\infty}
\gamma_k^{(d)}(c)a_k^{(d,m)}(c).
\]
By \eqref{eq-tail-aR-d2} and \eqref{eq-tail-aR-d3},
\[
R_j^{(d,m)}(c)\longrightarrow0
\qquad\text{as }j\to\infty.
\]
Hence, by \eqref{eq-tail-majorization},
\begin{equation}\label{eq-Wj-Rj}
\begin{aligned}
W_j
&\le
\sum_{k=j}^{\infty}
\left(
R_k^{(d,m)}(c)-R_{k+1}^{(d,m)}(c)
\right)
\\
&=
R_j^{(d,m)}(c),
\qquad
j\ge m+1.
\end{aligned}
\end{equation}

Set
\begin{equation}\label{eq-lambda-L63}
\lambda
=
\frac{
a_m^{(d,m)}(c)
}{
\mathbb A_m^{(d)}(c)
},
\end{equation}
and define
\begin{equation}\label{eq-rho-L63}
\rho_j
=
\begin{cases}
\lambda\,\mathbb A_j^{(d)}(c),
& 0\le j\le m,\\[1mm]
a_j^{(d,m)}(c),
& j\ge m+1.
\end{cases}
\end{equation}
By \eqref{criterion-Aj} and \eqref{eq-tail-aR-positive},
\[
\lambda>0,
\qquad
\rho_j>0,
\qquad
j\in\mathbb N_0.
\]
Moreover, by \eqref{eq-lambda-L63},
\begin{equation}\label{eq-rho-m-match}
\rho_m
=
\lambda\mathbb A_m^{(d)}(c)
=
a_m^{(d,m)}(c).
\end{equation}

For \(j\in\mathbb N_0\), set
\[
T_j
=
\sum_{k=0}^{\infty}
h_{\min\{j,k\}}^{(d)}
\gamma_k^{(d)}(c)\rho_k.
\]
This series is convergent. Indeed, by \eqref{eq-Wj-Rj} and \eqref{eq-rho-L63},
\[
\begin{aligned}
\sum_{k=m+1}^{\infty}
h_{\min\{j,k\}}^{(d)}
\gamma_k^{(d)}(c)\rho_k
&\le
h_j^{(d)}
\sum_{k=m+1}^{\infty}
\gamma_k^{(d)}(c)a_k^{(d,m)}(c)
\\
&=
h_j^{(d)}W_{m+1}
\\
&\le
h_j^{(d)}R_{m+1}^{(d,m)}(c)
<\infty.
\end{aligned}
\]

Let \(0\le j\le m\). By Lemma~\ref{L6.1},
\eqref{eq-rho-L63}, and the definition of \(W_{m+1}\),
\begin{equation}\label{eq-Tj-prefix}
\begin{aligned}
T_j
&=
\lambda
\sum_{k=0}^{m}
h_{\min\{j,k\}}^{(d)}
\gamma_k^{(d)}(c)
\mathbb A_k^{(d)}(c)
+
h_j^{(d)}W_{m+1}
\\
&=
\lambda\mathbb A_j^{(d)}(c)
+
h_j^{(d)}
\left(
W_{m+1}
-
\lambda\mathbb F_{m+1}^{(d)}(c)
\right)
\\
&=
\rho_j+
h_j^{(d)}
\left(
W_{m+1}
-
\lambda\mathbb F_{m+1}^{(d)}(c)
\right).
\end{aligned}
\end{equation}
By \eqref{eq-frakb-R-a}, \eqref{eq-finite-criterion}, and \eqref{eq-lambda-L63}, 
\[
\begin{aligned}
\lambda\mathbb F_{m+1}^{(d)}(c)
&>
\lambda\mathfrak b_m^{(d)}(c)
\mathbb A_m^{(d)}(c)
\\
&=
a_m^{(d,m)}(c)\mathfrak b_m^{(d)}(c)
\\
&=
R_{m+1}^{(d,m)}(c).
\end{aligned}
\]
Together with \eqref{eq-Wj-Rj}, this gives
\[
W_{m+1}
-
\lambda\mathbb F_{m+1}^{(d)}(c)
<0.
\]
Hence, by \eqref{eq-Tj-prefix},
\begin{equation}\label{eq-Tj-prefix-bound}
T_j<\rho_j,
\qquad
0\le j\le m.
\end{equation}

For \(j\ge m+1\),
\[
T_j-T_{j-1}
=
\varepsilon_j^{(d)}W_j.
\]
Using \eqref{eq-tail-increment} and \eqref{eq-Wj-Rj}, we obtain
\begin{equation}\label{eq-Tj-tail-increment}
T_j-T_{j-1}
\le
\varepsilon_j^{(d)}R_j^{(d,m)}(c)
=
a_j^{(d,m)}(c)-a_{j-1}^{(d,m)}(c).
\end{equation}
Since, by \eqref{eq-rho-m-match} and \eqref{eq-Tj-prefix-bound},
\[
T_m
<
\rho_m
=
a_m^{(d,m)}(c),
\]
it follows by induction from \eqref{eq-Tj-tail-increment} that
\[
T_j
<
a_j^{(d,m)}(c)
=
\rho_j,
\qquad
j\ge m+1.
\]
Together with \eqref{eq-Tj-prefix-bound}, this yields
\begin{equation}\label{eq-Tj-global-bound}
T_j<\rho_j,
\qquad
j\in\mathbb N_0.
\end{equation}

Finally, set
\[
p_j
=
\sqrt{\gamma_j^{(d)}(c)}\,\rho_j,
\qquad
j\in\mathbb N_0.
\]
Then \(p_j>0\) for every \(j\in\mathbb N_0\). Moreover, by
\eqref{eq-sjk} and \eqref{eq-Tj-global-bound},
\begin{equation}\label{eq-Sp-strict-p}
\sum_{k=0}^{\infty}
s_{j,k}^{(d)}(c)p_k
=
\sqrt{\gamma_j^{(d)}(c)}\,T_j
<
p_j,
\qquad
j\in\mathbb N_0.
\end{equation}
Arguing as in the proof of part (ii) of Lemma~\ref{L6.2}, we obtain
\[
\|S_c^{(d)}\|_{\ell^2\to\ell^2}\le1.
\]
By Theorem~\ref{thm:spectral-characterization} and Lemma~\ref{L5.9},
it follows that
\[
c\le c_d^*.
\]

Suppose, by contradiction, that
\[
c=c_d^*.
\]
By the proof of part (iii) of Lemma~\ref{L6.2}, there exists
\(\tau=(\tau_j)_{j\in\mathbb N_0}\in\ell^2(\mathbb N_0)\) such that
\begin{equation}\label{tau-j-pos}
\tau_j>0,
\qquad
j\in\mathbb N_0,
\end{equation}
and
\begin{equation}\label{tau-eq}
S_{c_d^*}^{(d)}\tau=\tau.
\end{equation}
Moreover, by 
\eqref{eq-tail-aR-d2}, \eqref{eq-tail-aR-d3}, and \eqref{eq-rho-L63}, 
the sequence \((\rho_j)_{j\in\mathbb N_0}\) is bounded. On the other
hand, by \eqref{eq-gammajd},
\[
\gamma_j^{(2)}(c)=O(j^{-3}),
\qquad
\gamma_j^{(3)}(c)=O(j^{-4})
\qquad\text{as }j\to\infty.
\]
Hence
\[
\sum_{j=0}^{\infty}p_j^2
=
\sum_{j=0}^{\infty}
\gamma_j^{(d)}(c)\rho_j^2
<\infty,
\]
and therefore
\[
p\in\ell^2(\mathbb N_0).
\]
By \eqref{eq-Sp-strict-p},
\[
p_j-\bigl(S_{c_d^*}^{(d)}p\bigr)_j>0,
\qquad
j\in\mathbb N_0.
\]
Together with \eqref{tau-j-pos}, this yields 
\[
\left\langle
\tau,p-S_{c_d^*}^{(d)}p
\right\rangle_{\ell^2}
=
\sum_{j=0}^{\infty}
\tau_j
\left(
p_j-\bigl(S_{c_d^*}^{(d)}p\bigr)_j
\right)
>0.
\]
On the other hand, since \(S_{c_d^*}^{(d)}\) is self-adjoint, by \eqref{tau-eq}, 
\[
\begin{aligned}
\left\langle
\tau,p-S_{c_d^*}^{(d)}p
\right\rangle_{\ell^2}
&=
\langle\tau,p\rangle_{\ell^2}
-
\left\langle
S_{c_d^*}^{(d)}\tau,p
\right\rangle_{\ell^2}
\\
&=0,
\end{aligned}
\]
which is a contradiction. Therefore,
\[
c<c_d^*.
\]
This completes the proof.
\end{proof}

For \(d\in\{2,3\}\), \(m\in\mathbb N_0\), and \(0<c\le c_d^*\), set
\begin{equation}\label{eq-Phi-dm}
\Phi_{d,m}(c)
=
\frac{
\mathbb F_{m+1}^{(d)}(c)
}{
\mathbb A_m^{(d)}(c)
}
-
\mathfrak b_m^{(d)}(c).
\end{equation}
By parts (i)--(ii) of Lemma~\ref{L6.2}, the function
\(\Phi_{d,m}\) is well defined on \((0,c_d^*]\).

\begin{lemma}\label{L6.4}
Let \(d\in\{2,3\}\) and \(m\in\mathbb N_0\). Then
\begin{equation}\label{eq-Phi-zero-limit}
\lim_{c\to0^+}\Phi_{d,m}(c)
=
\frac{1}{h_m^{(d)}}>0.
\end{equation}
Moreover, if \(0<c\le c_d^*\) and
\[
\Phi_{d,m}(c)=0,
\]
then
\begin{equation}\label{eq-Phi-next-positive}
\Phi_{d,m+1}(c)>0.
\end{equation}
\end{lemma}

\begin{proof}
By \eqref{eq-gammajd}, for every fixed \(j\in\mathbb N_0\),
\[
\gamma_j^{(d)}(c)\longrightarrow0
\qquad\text{as }c\to0^+.
\]
Moreover, by \eqref{eq-shooting-initial-d},
\[
\mathbb A_0^{(d)}(c)=1,
\qquad
\mathbb F_0^{(d)}(c)=\frac{1}{h_0^{(d)}}.
\]
Using \eqref{eq-epsilon-d}, \eqref{eq-shooting-F-d}, and
\eqref{eq-shooting-A-d}, a simultaneous induction on \(j\) gives, for
every fixed \(j\in\mathbb N_0\),
\[
\mathbb F_j^{(d)}(c)
\longrightarrow
\frac{1}{h_0^{(d)}},
\qquad
\mathbb A_j^{(d)}(c)
\longrightarrow
\frac{h_j^{(d)}}{h_0^{(d)}}
\qquad\text{as }c\to0^+.
\]
Hence,
\[
\frac{
\mathbb F_{m+1}^{(d)}(c)
}{
\mathbb A_m^{(d)}(c)
}
\longrightarrow
\frac{1}{h_m^{(d)}}
\qquad\text{as }c\to0^+.
\]
On the other hand, by \eqref{eq-frakb-d},
\[
\mathfrak b_m^{(d)}(c)\longrightarrow0
\qquad\text{as }c\to0^+.
\]
Therefore, letting \(c\to 0^+\) in \eqref{eq-Phi-dm}, we obtain \eqref{eq-Phi-zero-limit}.

Assume now that
\[
\Phi_{d,m}(c)=0
\]
for some \(0<c\le c_d^*\). By \eqref{eq-Phi-dm},
\begin{equation}\label{eq-Phi-zero-relation}
\mathbb F_{m+1}^{(d)}(c)
=
\mathfrak b_m^{(d)}(c)\mathbb A_m^{(d)}(c).
\end{equation}
Using \eqref{eq-shooting-A-d} and \eqref{eq-Phi-zero-relation}, we obtain
\begin{equation}\label{eq-Amplus1-Phi-zero}
\mathbb A_{m+1}^{(d)}(c)
=
\mathbb A_m^{(d)}(c)
\left(
1+\varepsilon_{m+1}^{(d)}
\mathfrak b_m^{(d)}(c)
\right).
\end{equation}
Similarly, by \eqref{eq-shooting-F-d},
\[
\mathbb F_{m+2}^{(d)}(c)
=
\mathbb F_{m+1}^{(d)}(c)
-
\gamma_{m+1}^{(d)}(c)
\mathbb A_{m+1}^{(d)}(c).
\]
Hence, by \eqref{eq-Phi-zero-relation} and
\eqref{eq-Amplus1-Phi-zero},
\[
\frac{
\mathbb F_{m+2}^{(d)}(c)
}{
\mathbb A_{m+1}^{(d)}(c)
}
=
\frac{
\mathfrak b_m^{(d)}(c)
}{
1+\varepsilon_{m+1}^{(d)}
\mathfrak b_m^{(d)}(c)
}
-
\gamma_{m+1}^{(d)}(c).
\]
Therefore, by \eqref{eq-Phi-dm},
\begin{equation}\label{eq-Phi-next-identity}
\Phi_{d,m+1}(c)
=
\frac{
\mathfrak b_m^{(d)}(c)
}{
1+\varepsilon_{m+1}^{(d)}
\mathfrak b_m^{(d)}(c)
}
-
\gamma_{m+1}^{(d)}(c)
-
\mathfrak b_{m+1}^{(d)}(c).
\end{equation}

For \(d=2\), using  \eqref{eq-gammajd}, \eqref{eq-epsilon-d-explicit}, 
\eqref{eq-frakb-d}, and \eqref{eq-Phi-next-identity}, a direct computation gives
\[
\Phi_{2,m+1}(c)
=
\frac{
16c^2(4m+9)
}{
(m+3)(2m+5)(4m+9-c)
\left(
16m^2+88m+117-(4m+11)c
\right)
}
>0.
\]
Indeed, \(0<c\le c_2^*<3\), so all the factors in the denominator
are strictly positive.

For \(d=3\),  we obtain
\[
\Phi_{3,m+1}(c)
=
\frac{
64c^2
}{
3(m+3)(2m+7)(4m+10-c)
\left(
24m^2+144m+210-(6m+17)c
\right)
}
>0.
\]
Here \(0<c\le c_3^*<4\), and all the factors in the denominator are
strictly positive. Thus \eqref{eq-Phi-next-positive} follows.
\end{proof}

We next derive quantitative estimates for \(\Phi_{d,m}(c_d^*)\).

\begin{lemma}\label{L6.5}
Let \(d\in\{2,3\}\) and \(m\in\mathbb N_0\). Then
\begin{equation}\label{eq-Phi-critical-negative}
\Phi_{d,m}(c_d^*)<0.
\end{equation}
Moreover, for every \(m\in\mathbb N\),
\begin{align}
0<-\Phi_{2,m}(c_2^*)
&<
\frac{9}{8m^3},
\label{eq-Phi-critical-d2}\\
0<-\Phi_{3,m}(c_3^*)
&<
\frac{3}{5m^4}.
\label{eq-Phi-critical-d3}
\end{align}
\end{lemma}

\begin{proof}
We first prove \eqref{eq-Phi-critical-negative}. By parts (i)--(ii) of
Lemma~\ref{L6.2},
\begin{equation}\label{eq-A-positive-critical}
\mathbb A_j^{(d)}(c_d^*)>0,
\qquad
j\in\mathbb N_0.
\end{equation}

Assume, by contradiction, that
\[
\Phi_{d,m}(c_d^*)\ge0.
\]
If \(\Phi_{d,m}(c_d^*)>0\), then \eqref{eq-Phi-dm} gives
\[
\mathbb F_{m+1}^{(d)}(c_d^*)
>
\mathfrak b_m^{(d)}(c_d^*)
\mathbb A_m^{(d)}(c_d^*).
\]
Together with \eqref{eq-A-positive-critical}, Lemma~\ref{L6.3}
implies \(c_d^*<c_d^*\), a contradiction.

If \(\Phi_{d,m}(c_d^*)=0\), Lemma~\ref{L6.4} yields
\[
\Phi_{d,m+1}(c_d^*)>0.
\]
Hence, by \eqref{eq-Phi-dm},
\[
\mathbb F_{m+2}^{(d)}(c_d^*)
>
\mathfrak b_{m+1}^{(d)}(c_d^*)
\mathbb A_{m+1}^{(d)}(c_d^*).
\]
Applying Lemma~\ref{L6.3} with \(m+1\) in place of \(m\), together
with \eqref{eq-A-positive-critical}, again gives \(c_d^*<c_d^*\), a
contradiction.

Therefore,
\[
\Phi_{d,m}(c_d^*)<0,
\]
which proves \eqref{eq-Phi-critical-negative}.

We next prove \eqref{eq-Phi-critical-d2}. We first claim that
\begin{equation}\label{eq-coarse-c2}
\frac53<c_2^*<\frac74.
\end{equation}
At \(c=5/3\), \eqref{eq-shooting-initial-d},
\eqref{eq-shooting-F-d}, and \eqref{eq-frakb-d} give
\[
\mathbb A_0^{(2)}\left(\frac53\right)=1,
\qquad
\mathbb F_1^{(2)}\left(\frac53\right)
-
\mathfrak b_0^{(2)}\left(\frac53\right)
=
\frac13>0.
\]
Hence Lemma~\ref{L6.3} yields
\[
\frac53<c_2^*.
\]
On the other hand, by Corollary~\ref{cor:Hd-d23} and Theorem~C,
\[
c_2^*
=
2\mathcal H_2
\le
\sqrt3
<
\frac74.
\]
Thus \eqref{eq-coarse-c2} follows.

By \eqref{eq-Ftilde-def} and \eqref{eq-identification-tilde-AF}, for
every \(m\in\mathbb N\),
\begin{equation}\label{eq-critical-tail-d2}
\mathbb F_{m+1}^{(2)}(c_2^*)
=
\sum_{k=m+1}^{\infty}
\gamma_k^{(2)}(c_2^*)
\mathbb A_k^{(2)}(c_2^*).
\end{equation}
Moreover, part (ii) of Lemma~\ref{L6.2}, together with
\eqref{exp-epsilon-md} and \eqref{eq-shooting-A-d}, shows that
\[
\mathbb A_{j+1}^{(2)}(c_2^*)
>
\mathbb A_j^{(2)}(c_2^*),
\qquad
j\in\mathbb N_0.
\]
Hence \eqref{eq-critical-tail-d2} gives
\begin{equation}\label{eq-critical-ratio-lower-d2}
\frac{
\mathbb F_{m+1}^{(2)}(c_2^*)
}{
\mathbb A_m^{(2)}(c_2^*)
}
>
\sum_{k=m+1}^{\infty}
\gamma_k^{(2)}(c_2^*).
\end{equation}

By \eqref{eq-gammajd},
\[
\frac{\gamma_k^{(2)}(c)}{c}
=
\frac{
8(4k+5)^2
}{
(k+1)(k+2)(2k+1)(2k+3)(4k+5-c)
},
\qquad
0<c<3,
\]
which is strictly increasing in \(c\). Hence, by \eqref{eq-coarse-c2},
\[
\frac{\gamma_k^{(2)}(c_2^*)}{c_2^*}
>
\frac{\gamma_k^{(2)}(5/3)}{5/3}.
\]
Moreover,
\[
\frac{\gamma_k^{(2)}(5/3)}{5/3}
-
\frac{8}{(k+10/9)^3}
=
\frac{
12\left(
7965k^3+26800k^2+29254k+10420
\right)
}{
(k+1)(k+2)(2k+1)(2k+3)(6k+5)(9k+10)^3
}
>0.
\]
Therefore,
\begin{equation}\label{eq-gamma-lower-d2}
\gamma_k^{(2)}(c_2^*)
>
\frac{8c_2^*}{(k+10/9)^3},
\qquad
k\in\mathbb N_0.
\end{equation}

Since \(x\mapsto(a+x)^{-3}\) is strictly convex for \(a>0\), the
trapezoidal inequality yields
\[
\int_0^\infty \frac{dx}{(a+x)^3}
<
\frac{1}{2a^3}
+
\sum_{n=1}^{\infty}\frac{1}{(a+n)^3}.
\]
Hence
\[
\sum_{n=0}^{\infty}\frac{1}{(a+n)^3}
>
\frac{1}{2a^2}
+
\frac{1}{2a^3}.
\]
Taking
\[
a=m+\frac{19}{9}
\]
and using \eqref{eq-gamma-lower-d2}, we obtain
\begin{equation}\label{eq-gamma-tail-lower-d2}
\sum_{k=m+1}^{\infty}\gamma_k^{(2)}(c_2^*)
>
8c_2^*
\left[
\frac{1}{2(m+19/9)^2}
+
\frac{1}{2(m+19/9)^3}
\right].
\end{equation}
Combining \eqref{eq-critical-ratio-lower-d2} and
\eqref{eq-gamma-tail-lower-d2}, we obtain
\begin{equation}\label{ineq-FA}
\frac{
\mathbb F_{m+1}^{(2)}(c_2^*)
}{
\mathbb A_m^{(2)}(c_2^*)
}
>
8c_2^*
\left[
\frac{1}{2(m+19/9)^2}
+
\frac{1}{2(m+19/9)^3}
\right].
\end{equation}

Set
\[
D_{2,m}(c)
=
\mathfrak b_m^{(2)}(c)
-
8c
\left[
\frac{1}{2(m+19/9)^2}
+
\frac{1}{2(m+19/9)^3}
\right],
\qquad
0<c<3.
\]
By \eqref{eq-Phi-dm}, \eqref{eq-Phi-critical-negative}, and
\eqref{ineq-FA},
\begin{equation}\label{eq-Phi-D2}
0<-\Phi_{2,m}(c_2^*)<D_{2,m}(c_2^*).
\end{equation}
Moreover,
\[
\frac{\mathfrak b_m^{(2)}(c)}{c}
=
\frac{
8(4m+5)(4m+9)
}{
(m+2)(2m+3)
\left(
16m^2+56m+45-(4m+7)c
\right)
},
\qquad
0<c<3,
\]
is strictly increasing in \(c\). Hence
\[
c\longmapsto \frac{D_{2,m}(c)}{c}
\]
is strictly increasing on \((0,3)\). By \eqref{eq-coarse-c2} and
\eqref{eq-Phi-D2},
\[
0<
\frac{D_{2,m}(c_2^*)}{c_2^*}
<
\frac{D_{2,m}(7/4)}{7/4}.
\]
Therefore,
\begin{equation}\label{eq-D2-upper}
0<D_{2,m}(c_2^*)
<
\frac{c_2^*}{7/4}
D_{2,m}\left(\frac74\right)
<
D_{2,m}\left(\frac74\right).
\end{equation}

Finally, a direct computation gives
\[
\frac{9}{8m^3}
-
D_{2,m}\left(\frac74\right)
=
\frac{P_2(m)}
{
8m^3(m+2)(2m+3)(9m+19)^3
(64m^2+196m+131)
},
\]
where
\[
\begin{aligned}
P_2(m)
={}&
m^5
\left(
5184m^2-494856m+14343170
\right)
\\
&+
108512311m^4
+
282557277m^3
+
340838055m^2
\\
&+
198153261m
+
48520566.
\end{aligned}
\]
The quadratic polynomial
\[
5184m^2-494856m+14343170
\]
is strictly positive on \(\mathbb R\), since its leading coefficient is
positive and its discriminant satisfies
\[
(-494856)^2
-
4(5184)(14343170)
=
-52537512384<0.
\]
Since all the remaining coefficients of \(P_2\) are strictly positive,
we have
\[
P_2(m)>0,
\qquad
m\in\mathbb N.
\]
Therefore,
\begin{equation}\label{eq-D2-final}
D_{2,m}\left(\frac74\right)
<
\frac{9}{8m^3},
\qquad
m\in\mathbb N.
\end{equation}
Combining \eqref{eq-Phi-D2}, \eqref{eq-D2-upper}, and
\eqref{eq-D2-final} yields \eqref{eq-Phi-critical-d2}.

We finally prove \eqref{eq-Phi-critical-d3}. As in the case \(d=2\),
we first establish
\begin{equation}\label{eq-coarse-c3}
\frac{18}{5}<c_3^*<\frac{11}{3}.
\end{equation}
Indeed, at \(c=18/5\), \eqref{eq-shooting-initial-d},
\eqref{eq-shooting-F-d}, and \eqref{eq-frakb-d} give
\[
\mathbb A_0^{(3)}\left(\frac{18}{5}\right)=1,
\qquad
\mathbb F_1^{(3)}\left(\frac{18}{5}\right)
-
\mathfrak b_0^{(3)}\left(\frac{18}{5}\right)
=
\frac{8}{35}>0.
\]
Hence Lemma~\ref{L6.3} yields
\[
\frac{18}{5}<c_3^*.
\]
On the other hand, Corollary~\ref{cor:Hd-d23} and Theorem~C give
\[
c_3^*
=
2\mathcal H_3
\le
\frac{4\sqrt{21}}{5}
<
\frac{11}{3},
\]
and \eqref{eq-coarse-c3} follows.

By \eqref{eq-Ftilde-def} and \eqref{eq-identification-tilde-AF}, for
every \(m\in\mathbb N\),
\begin{equation}\label{eq-critical-tail-d3}
\mathbb F_{m+1}^{(3)}(c_3^*)
=
\sum_{k=m+1}^{\infty}
\gamma_k^{(3)}(c_3^*)
\mathbb A_k^{(3)}(c_3^*).
\end{equation}
As in the case \(d=2\), part (ii) of Lemma~\ref{L6.2}, together with
\eqref{exp-epsilon-md} and \eqref{eq-shooting-A-d}, yields
\[
\mathbb A_{j+1}^{(3)}(c_3^*)
>
\mathbb A_j^{(3)}(c_3^*),
\qquad
j\in\mathbb N_0.
\]
Hence \eqref{eq-critical-tail-d3} gives
\begin{equation}\label{eq-critical-ratio-lower-d3}
\frac{
\mathbb F_{m+1}^{(3)}(c_3^*)
}{
\mathbb A_m^{(3)}(c_3^*)
}
>
\sum_{k=m+1}^{\infty}
\gamma_k^{(3)}(c_3^*).
\end{equation}

By \eqref{eq-gammajd},
\[
\frac{\gamma_k^{(3)}(c)}{c}
=
\frac{
16(2k+3)
}{
(k+1)(k+2)(2k+1)(2k+5)(4k+6-c)
},
\qquad
0<c<4,
\]
which is strictly increasing in \(c\). Hence, by
\eqref{eq-coarse-c3},
\[
\frac{\gamma_k^{(3)}(c_3^*)}{c_3^*}
>
\frac{\gamma_k^{(3)}(18/5)}{18/5}.
\]
Moreover,
\[
\frac{\gamma_k^{(3)}(18/5)}{18/5}
-
\frac{2}{(k+51/40)^4}
=
\frac{
20\left(
7955200k^3+29805120k^2+34450882k+12615603
\right)
}{
(k+1)(k+2)(2k+1)(2k+5)(5k+3)(40k+51)^4
}
>0.
\]
Therefore,
\begin{equation}\label{eq-gamma-lower-d3}
\gamma_k^{(3)}(c_3^*)
>
\frac{2c_3^*}{(k+51/40)^4},
\qquad
k\in\mathbb N_0.
\end{equation}

Since \(x\mapsto(a+x)^{-4}\) is strictly convex for \(a>0\), the
trapezoidal inequality yields
\[
\int_0^\infty \frac{dx}{(a+x)^4}
<
\frac{1}{2a^4}
+
\sum_{n=1}^{\infty}\frac{1}{(a+n)^4}.
\]
Hence
\[
\sum_{n=0}^{\infty}\frac{1}{(a+n)^4}
>
\frac{1}{3a^3}
+
\frac{1}{2a^4}.
\]
Taking
\[
a=m+\frac{91}{40}
\]
and using \eqref{eq-gamma-lower-d3}, we obtain
\begin{equation}\label{eq-gamma-tail-lower-d3}
\sum_{k=m+1}^{\infty}\gamma_k^{(3)}(c_3^*)
>
2c_3^*
\left[
\frac{1}{3(m+91/40)^3}
+
\frac{1}{2(m+91/40)^4}
\right].
\end{equation}
Combining \eqref{eq-critical-ratio-lower-d3} and
\eqref{eq-gamma-tail-lower-d3}, we obtain
\begin{equation}\label{ineq-FA-d3}
\frac{
\mathbb F_{m+1}^{(3)}(c_3^*)
}{
\mathbb A_m^{(3)}(c_3^*)
}
>
2c_3^*
\left[
\frac{1}{3(m+91/40)^3}
+
\frac{1}{2(m+91/40)^4}
\right].
\end{equation}

Set
\[
D_{3,m}(c)
=
\mathfrak b_m^{(3)}(c)
-
2c
\left[
\frac{1}{3(m+91/40)^3}
+
\frac{1}{2(m+91/40)^4}
\right],
\qquad
0<c<4.
\]
By \eqref{eq-Phi-dm}, \eqref{eq-Phi-critical-negative}, and
\eqref{ineq-FA-d3},
\begin{equation}\label{eq-Phi-D3}
0<-\Phi_{3,m}(c_3^*)<D_{3,m}(c_3^*).
\end{equation}

Moreover,
\[
\frac{\mathfrak b_m^{(3)}(c)}{c}
=
\frac{
16
}{
(m+2)
\left(
24m^2+96m+90-(6m+11)c
\right)
},
\qquad
0<c<4,
\]
is strictly increasing in \(c\). Hence
\[
c\longmapsto\frac{D_{3,m}(c)}{c}
\]
is strictly increasing on \((0,4)\). By \eqref{eq-coarse-c3} and
\eqref{eq-Phi-D3},
\[
0<
\frac{D_{3,m}(c_3^*)}{c_3^*}
<
\frac{D_{3,m}(11/3)}{11/3}.
\]
Consequently,
\begin{equation}\label{eq-D3-upper}
0<D_{3,m}(c_3^*)
<
\frac{c_3^*}{11/3}
D_{3,m}\left(\frac{11}{3}\right)
<
D_{3,m}\left(\frac{11}{3}\right).
\end{equation}

Finally, a direct computation gives
\[
\frac{3}{5m^4}
-
D_{3,m}\left(\frac{11}{3}\right)
=
\frac{P_3(m)}
{
45m^4(m+2)(40m+91)^4
(72m^2+222m+149)
},
\]
where
\[
\begin{aligned}
P_3(m)
={}&
m^5
\left(
76800000m^2
-
2974208000m
+
185116915200
\right)
\\
&+
1187248251920m^4
+
2785071593304m^3
+
3247763499162m^2
\\
&+
2068070863131m
+
551754136206.
\end{aligned}
\]
The quadratic polynomial
\[
76800000m^2
-
2974208000m
+
185116915200
\]
is strictly positive on \(\mathbb R\), since its leading coefficient is
positive and its discriminant satisfies
\[
(-2974208000)^2
-
4(76800000)(185116915200)
=
-48022003122176000000<0.
\]
Since all the remaining coefficients of \(P_3\) are strictly positive,
we have
\[
P_3(m)>0,
\qquad
m\in\mathbb N.
\]
Consequently,
\begin{equation}\label{eq-D3-final}
D_{3,m}\left(\frac{11}{3}\right)
<
\frac{3}{5m^4},
\qquad
m\in\mathbb N.
\end{equation}
Combining \eqref{eq-Phi-D3}, \eqref{eq-D3-upper}, and
\eqref{eq-D3-final} yields \eqref{eq-Phi-critical-d3}.
This completes the proof.
\end{proof}

We next establish the monotonicity of \(\Phi_{d,m}\).

\begin{lemma}\label{L6.6}
Let \(d\in\{2,3\}\) and \(m\in\mathbb N_0\). Then
\(\Phi_{d,m}\) is continuously differentiable and strictly decreasing
on \((0,c_d^*]\). More precisely,
\[
\Phi_{d,m}'(c)
=
-
\frac{
\displaystyle
\sum_{j=0}^{m}
\bigl(\gamma_j^{(d)}\bigr)'(c)
\bigl(\mathbb A_j^{(d)}(c)\bigr)^2
}{
\bigl(\mathbb A_m^{(d)}(c)\bigr)^2
}
-
\bigl(\mathfrak b_m^{(d)}\bigr)'(c)
<0,
\qquad
0<c\le c_d^*.
\]
\end{lemma}

\begin{proof}
Fix \(d\in\{2,3\}\) and \(m\in\mathbb N_0\). For each
\(j\in\mathbb N_0\), the function \(\gamma_j^{(d)}\) is smooth on
\((0,d+1)\). Hence, by the recursions
\eqref{eq-shooting-F-d}--\eqref{eq-shooting-A-d},
the functions \(\mathbb A_j^{(d)}\) and \(\mathbb F_j^{(d)}\) are
continuously differentiable on \((0,d+1)\).

Moreover, by part (ii) of Lemma~\ref{L6.2},
\[
\mathbb F_j^{(d)}(c)>0,
\qquad
j\in\mathbb N_0,\quad 0<c\le c_d^*.
\]
Therefore, \eqref{eq:A-F-representation} gives
\[
\mathbb A_j^{(d)}(c)>0,
\qquad
j\in\mathbb N_0,\quad 0<c\le c_d^*.
\]
Since \(\mathfrak b_m^{(d)}\) is smooth on \((0,d+1)\), it follows from
\eqref{eq-Phi-dm} that \(\Phi_{d,m}\) is continuously differentiable on
\((0,c_d^*]\).

To simplify the notation, throughout the remainder of the proof we write
\[
\mathbb A_j^{(d)}
=
\mathbb A_j^{(d)}(c),
\qquad
\mathbb F_j^{(d)}
=
\mathbb F_j^{(d)}(c),
\qquad
\gamma_j^{(d)}
=
\gamma_j^{(d)}(c),
\qquad
0<c\le c_d^*,
\]
and denote differentiation with respect to \(c\) by a prime.

Set
\[
W_j
=
\bigl(\mathbb F_{j+1}^{(d)}\bigr)'
\mathbb A_j^{(d)}
-
\mathbb F_{j+1}^{(d)}
\bigl(\mathbb A_j^{(d)}\bigr)',
\qquad
j\in\mathbb N_0.
\]
Since
\[
\mathbb A_0^{(d)}=1,
\qquad
\mathbb F_0^{(d)}=\frac{1}{h_0^{(d)}},
\]
are independent of \(c\), \eqref{eq-shooting-F-d} gives
\begin{equation}\label{W0-eq}
W_0
=
-
\bigl(\gamma_0^{(d)}\bigr)'
\bigl(\mathbb A_0^{(d)}\bigr)^2.
\end{equation}
For \(j\in\mathbb N\), differentiating \eqref{eq-shooting-F-d} with
respect to \(c\) gives
\[
\bigl(\mathbb F_{j+1}^{(d)}\bigr)'
=
\bigl(\mathbb F_j^{(d)}\bigr)'
-
\bigl(\gamma_j^{(d)}\bigr)'
\mathbb A_j^{(d)}
-
\gamma_j^{(d)}
\bigl(\mathbb A_j^{(d)}\bigr)'.
\]
Therefore,
\[
\begin{aligned}
W_j
&=
\left[
\bigl(\mathbb F_j^{(d)}\bigr)'
-
\bigl(\gamma_j^{(d)}\bigr)'
\mathbb A_j^{(d)}
-
\gamma_j^{(d)}
\bigl(\mathbb A_j^{(d)}\bigr)'
\right]
\mathbb A_j^{(d)}
\\
&\quad
-
\left(
\mathbb F_j^{(d)}
-
\gamma_j^{(d)}\mathbb A_j^{(d)}
\right)
\bigl(\mathbb A_j^{(d)}\bigr)'
\\
&=
\bigl(\mathbb F_j^{(d)}\bigr)'
\mathbb A_j^{(d)}
-
\mathbb F_j^{(d)}
\bigl(\mathbb A_j^{(d)}\bigr)'
-
\bigl(\gamma_j^{(d)}\bigr)'
\bigl(\mathbb A_j^{(d)}\bigr)^2.
\end{aligned}
\]
On the other hand, \eqref{eq-shooting-A-d} yields
\[
\mathbb A_j^{(d)}
=
\mathbb A_{j-1}^{(d)}
+
\varepsilon_j^{(d)}\mathbb F_j^{(d)},
\]
and hence
\[
\bigl(\mathbb A_j^{(d)}\bigr)'
=
\bigl(\mathbb A_{j-1}^{(d)}\bigr)'
+
\varepsilon_j^{(d)}
\bigl(\mathbb F_j^{(d)}\bigr)'.
\]
It follows that
\[
\begin{aligned}
\bigl(\mathbb F_j^{(d)}\bigr)'
\mathbb A_j^{(d)}
-
\mathbb F_j^{(d)}
\bigl(\mathbb A_j^{(d)}\bigr)'
&=
\bigl(\mathbb F_j^{(d)}\bigr)'
\mathbb A_{j-1}^{(d)}
-
\mathbb F_j^{(d)}
\bigl(\mathbb A_{j-1}^{(d)}\bigr)'
\\
&=
W_{j-1}.
\end{aligned}
\]
Consequently,
\[
W_j
=
W_{j-1}
-
\bigl(\gamma_j^{(d)}\bigr)'
\bigl(\mathbb A_j^{(d)}\bigr)^2,
\qquad
j\in\mathbb N.
\]
Together with \eqref{W0-eq}, this yields, by induction,
\[
W_m
=
-
\sum_{j=0}^{m}
\bigl(\gamma_j^{(d)}\bigr)'
\bigl(\mathbb A_j^{(d)}\bigr)^2.
\]
Since \(\mathbb A_m^{(d)}>0\), we conclude that
\[
\begin{aligned}
\frac{d}{dc}
\left(
\frac{
\mathbb F_{m+1}^{(d)}
}{
\mathbb A_m^{(d)}
}
\right)
&=
\frac{
\bigl(\mathbb F_{m+1}^{(d)}\bigr)'
\mathbb A_m^{(d)}
-
\mathbb F_{m+1}^{(d)}
\bigl(\mathbb A_m^{(d)}\bigr)'
}{
\bigl(\mathbb A_m^{(d)}\bigr)^2
}
\\
&=
\frac{W_m}{\bigl(\mathbb A_m^{(d)}\bigr)^2}
\\
&=
-
\frac{
\displaystyle
\sum_{j=0}^{m}
\bigl(\gamma_j^{(d)}\bigr)'
\bigl(\mathbb A_j^{(d)}\bigr)^2
}{
\bigl(\mathbb A_m^{(d)}\bigr)^2
}.
\end{aligned}
\]

Returning to the original notation and differentiating
\eqref{eq-Phi-dm}, we obtain
\[
\Phi_{d,m}'(c)
=
-
\frac{
\displaystyle
\sum_{j=0}^{m}
\bigl(\gamma_j^{(d)}\bigr)'(c)
\bigl(\mathbb A_j^{(d)}(c)\bigr)^2
}{
\bigl(\mathbb A_m^{(d)}(c)\bigr)^2
}
-
\bigl(\mathfrak b_m^{(d)}\bigr)'(c),
\qquad
0<c\le c_d^*.
\]
Finally, \eqref{eq-gammajd} and \eqref{eq-frakb-d} show that
\begin{equation}\label{sign-dervs}
\bigl(\gamma_j^{(d)}\bigr)'(c)>0,
\qquad
\bigl(\mathfrak b_m^{(d)}\bigr)'(c)>0,
\qquad
0\le j\le m,
\qquad 
0<c<d+1.
\end{equation}
Hence
\[
\Phi_{d,m}'(c)<0,
\qquad
0<c\le c_d^*,
\]
which completes the proof.
\end{proof}

\begin{lemma}\label{L6.7}
For each \(d\in\{2,3\}\) and \(m\in\mathbb N_0\), the function
\(\Phi_{d,m}\) has a unique zero
\[
\underline c_{d,m}\in(0,c_d^*).
\]
Moreover, \(\underline c_{d,m}\) is the smallest solution of
\eqref{eq-cdm-def} in \((0,d+1)\), and
\begin{equation}\label{eq-cdm-monotone}
\underline c_{d,m}
<
\underline c_{d,m+1}
<
c_d^*.
\end{equation}
In addition,
\begin{equation}\label{eq-cdm-coarse-lower}
\frac53<\underline c_{2,m},
\qquad
\frac{18}{5}<\underline c_{3,m}.
\end{equation}
\end{lemma}

\begin{proof}
Fix \(d\in\{2,3\}\) and \(m\in\mathbb N_0\). By
Lemma~\ref{L6.4},
\[
\lim_{c\to0^+}\Phi_{d,m}(c)
=
\frac{1}{h_m^{(d)}}>0,
\]
whereas Lemma~\ref{L6.5} gives
\[
\Phi_{d,m}(c_d^*)<0.
\]
Since \(\Phi_{d,m}\) is continuous on \((0,c_d^*]\) by
Lemma~\ref{L6.6}, the intermediate value theorem ensures the existence of
\[
\underline c_{d,m}\in(0,c_d^*)
\]
such that
\begin{equation}\label{zero-Phi-d-m}
\Phi_{d,m}(\underline c_{d,m})=0.
\end{equation}
Moreover, Lemma~\ref{L6.6} shows that \(\Phi_{d,m}\) is strictly
decreasing on \((0,c_d^*]\). Hence this zero is unique in
\((0,c_d^*)\).

By \eqref{eq-Phi-dm}, \eqref{zero-Phi-d-m} is equivalent to
\[
\mathbb F_{m+1}^{(d)}(\underline c_{d,m})
=
\mathfrak b_m^{(d)}(\underline c_{d,m})
\mathbb A_m^{(d)}(\underline c_{d,m}),
\]
so \(\underline c_{d,m}\) is a solution of \eqref{eq-cdm-def}.
Furthermore, the strict monotonicity of \(\Phi_{d,m}\) gives
\[
\Phi_{d,m}(c)>0,
\qquad
0<c<\underline c_{d,m},
\]
and hence \eqref{eq-cdm-def} has no solution in
\((0,\underline c_{d,m})\). Since
\[
\underline c_{d,m}\in(0,c_d^*)\subset(0,d+1),
\]
it follows that \(\underline c_{d,m}\) is the smallest solution of
\eqref{eq-cdm-def} in \((0,d+1)\).

By \eqref{zero-Phi-d-m} and Lemma~\ref{L6.4},
\[
\Phi_{d,m+1}(\underline c_{d,m})>0.
\]
On the other hand, applying \eqref{zero-Phi-d-m} with \(m+1\) in place
of \(m\), we have
\[
\Phi_{d,m+1}(\underline c_{d,m+1})=0.
\]
Since \(\Phi_{d,m+1}\) is strictly decreasing on \((0,c_d^*]\) by
Lemma~\ref{L6.6}, it follows that
\[
\underline c_{d,m}
<
\underline c_{d,m+1}
<
c_d^*,
\]
which proves \eqref{eq-cdm-monotone}.

By \eqref{eq-Phi-dm}, a direct computation gives
\[
\Phi_{2,0}\left(\frac53\right)=\frac13>0,
\qquad
\Phi_{3,0}\left(\frac{18}{5}\right)=\frac{8}{35}>0.
\]
Since \(\Phi_{2,0}\) and \(\Phi_{3,0}\) are strictly decreasing and vanish
at \(\underline c_{2,0}\) and \(\underline c_{3,0}\), respectively, we obtain
\[
\frac53<\underline c_{2,0},
\qquad
\frac{18}{5}<\underline c_{3,0}.
\]
Together with \eqref{eq-cdm-monotone}, this yields
\eqref{eq-cdm-coarse-lower}. This completes the proof.
\end{proof}

We next derive uniform bounds for the recursive coefficients and for the
derivatives of \(\Phi_{d,m}\).

\begin{lemma}\label{L6.8}
Let \(m\in\mathbb N\). Then
\begin{align}
\mathbb A_m^{(2)}(c)
&<
\frac43,
\qquad
\underline c_{2,m}<c<c_2^*,
\label{eq-Am-upper-d2}\\
\mathbb A_m^{(3)}(c)
&<
\frac32,
\qquad
\underline c_{3,m}<c<c_3^*.
\label{eq-Am-upper-d3}
\end{align}
Consequently,
\begin{align}
-\Phi_{2,m}'(c)
&>
\frac{135}{16},
\qquad
\underline c_{2,m}<c<c_2^*,
\label{eq-Phi-slope-d2}\\
-\Phi_{3,m}'(c)
&>
\frac{20}{9},
\qquad
\underline c_{3,m}<c<c_3^*.
\label{eq-Phi-slope-d3}
\end{align}
\end{lemma}

\begin{proof}
We prove \eqref{eq-Am-upper-d2} and \eqref{eq-Am-upper-d3}
simultaneously. Let \(d\in\{2,3\}\) and
\[
\underline c_{d,m}<c<c_d^*.
\]
For \(0\le j\le m-1\), Lemma~\ref{L6.7} gives
\[
\underline c_{d,j}
\le
\underline c_{d,m}
<
c.
\]
Since \(\Phi_{d,j}\) is strictly decreasing on \((0,c_d^*]\) and
\[
\Phi_{d,j}(\underline c_{d,j})=0,
\]
we have
\[
\Phi_{d,j}(c)<0.
\]
Hence, by \eqref{eq-Phi-dm},
\[
\mathbb F_{j+1}^{(d)}(c)
<
\mathfrak b_j^{(d)}(c)\mathbb A_j^{(d)}(c).
\]
Using \eqref{eq-shooting-A-d}, we obtain
\begin{equation}\label{iterative-ineq-d23}
\mathbb A_{j+1}^{(d)}(c)
<
\mathbb A_j^{(d)}(c)
\left(
1+
\varepsilon_{j+1}^{(d)}
\mathfrak b_j^{(d)}(c)
\right),
\qquad
0\le j\le m-1.
\end{equation}

Set
\[
\kappa_2=4,
\qquad
\kappa_3=3.
\]
By \eqref{eq-epsilon-d-explicit}, \eqref{eq-frakb-d},
\eqref{eq-coarse-c2}, and \eqref{eq-coarse-c3}, a direct computation
gives
\begin{equation}\label{eq-eps-b-unified}
\varepsilon_{j+1}^{(d)}
\mathfrak b_j^{(d)}(c)
<
\frac{1}{\kappa_d(j+1)(j+2)},
\qquad
0\le j\le m-1.
\end{equation}
Combining \eqref{iterative-ineq-d23} and
\eqref{eq-eps-b-unified}, and using
\(\mathbb A_0^{(d)}(c)=1\), yields
\[
\mathbb A_m^{(d)}(c)
<
\prod_{j=0}^{m-1}
\left(
1+\frac{1}{\kappa_d(j+1)(j+2)}
\right).
\]
Since \(1+x<e^x\) for \(x>0\),
\[
\begin{aligned}
\mathbb A_m^{(d)}(c)
&<
\exp\left(
\frac{1}{\kappa_d}
\sum_{j=0}^{m-1}
\frac{1}{(j+1)(j+2)}
\right)
\\
&=
\exp\left(
\frac{1}{\kappa_d}
\left(1-\frac{1}{m+1}\right)
\right)
<
e^{1/\kappa_d}.
\end{aligned}
\]
Therefore,
\[
\mathbb A_m^{(2)}(c)
<
e^{1/4}
<
\frac43,
\qquad
\mathbb A_m^{(3)}(c)
<
e^{1/3}
<
\frac32,
\]
which proves \eqref{eq-Am-upper-d2} and
\eqref{eq-Am-upper-d3}.

It remains to prove \eqref{eq-Phi-slope-d2} and
\eqref{eq-Phi-slope-d3}. Let \(d\in\{2,3\}\). By Lemma~\ref{L6.6} and
\eqref{sign-dervs}, all the terms in the expression of
\(-\Phi_{d,m}'(c)\) are positive. Hence, using
\(\mathbb A_0^{(d)}(c)=1\), we obtain
\begin{equation}\label{eq-Phi-derivative-j0}
-\Phi_{d,m}'(c)
>
\frac{
\bigl(\gamma_0^{(d)}\bigr)'(c)
}{
\bigl(\mathbb A_m^{(d)}(c)\bigr)^2
},
\qquad
\underline c_{d,m}<c<c_d^*.
\end{equation}

For \(d=2\) and \(\underline c_{2,m}<c<c_2^*\),
\eqref{eq-gammajd} gives
\[
\bigl(\gamma_0^{(2)}\bigr)'(c)
=
\frac{500}{3(5-c)^2},
\]
which is strictly increasing on \((0,3)\). By Lemma~\ref{L6.7},
\[
c>\underline c_{2,m}>\frac53,
\]
and therefore
\[
\bigl(\gamma_0^{(2)}\bigr)'(c)
>
\bigl(\gamma_0^{(2)}\bigr)'\left(\frac53\right)
=
15.
\]
Combining this with \eqref{eq-Am-upper-d2} and
\eqref{eq-Phi-derivative-j0}, we obtain
\[
-\Phi_{2,m}'(c)
>
15\left(\frac34\right)^2
=
\frac{135}{16},
\]
which proves \eqref{eq-Phi-slope-d2}.

For \(d=3\) and \(\underline c_{3,m}<c<c_3^*\),
\eqref{eq-gammajd} gives
\[
\bigl(\gamma_0^{(3)}\bigr)'(c)
=
\frac{144}{5(6-c)^2},
\]
which is strictly increasing on \((0,4)\). By Lemma~\ref{L6.7},
\[
c>\underline c_{3,m}>\frac{18}{5},
\]
and therefore
\[
\bigl(\gamma_0^{(3)}\bigr)'(c)
>
\bigl(\gamma_0^{(3)}\bigr)'\left(\frac{18}{5}\right)
=
5.
\]
Combining this with \eqref{eq-Am-upper-d3} and
\eqref{eq-Phi-derivative-j0}, we obtain
\[
-\Phi_{3,m}'(c)
>
5\left(\frac23\right)^2
=
\frac{20}{9},
\]
which proves \eqref{eq-Phi-slope-d3}. This completes the proof.
\end{proof}

\subsection{Proof of Theorem~\ref{thm:quantitative-estimate}}

\begin{proof}
(i) It follows directly from Lemma~\ref{L6.7}.

\medskip

\noindent (iii) Fix \(m\in\mathbb N\) and \(d\in\{2,3\}\). By
Lemma~\ref{L6.6}, \(\Phi_{d,m}\) is continuously differentiable and
strictly decreasing on \((0,c_d^*]\). Hence, by
\eqref{zero-Phi-d-m} and the mean value theorem, there exists
\[
\xi_{d,m}\in(\underline c_{d,m},c_d^*)
\]
such that
\[
\Phi_{d,m}(c_d^*)
=
\Phi_{d,m}'(\xi_{d,m})
\bigl(c_d^*-\underline c_{d,m}\bigr),
\]
with
\[
\Phi_{d,m}'(\xi_{d,m})<0.
\]
Since \(\underline c_{d,m}<c_d^*\), we obtain
\begin{equation}\label{eq-error-MVT}
0<c_d^*-\underline c_{d,m}
=
\frac{
-\Phi_{d,m}(c_d^*)
}{
-\Phi_{d,m}'(\xi_{d,m})
}.
\end{equation}

For \(d=2\), Lemmas~\ref{L6.5} and~\ref{L6.8} give
\[
0<-\Phi_{2,m}(c_2^*)
<
\frac{9}{8m^3}
\]
and
\[
-\Phi_{2,m}'(\xi_{2,m})
>
\frac{135}{16}.
\]
Hence, by \eqref{eq-error-MVT},
\[
0<c_2^*-\underline c_{2,m}
<
\frac{9}{8m^3}\frac{16}{135}
=
\frac{2}{15m^3},
\]
which proves \eqref{eq-c2-effective-error}.

Similarly, for \(d=3\), Lemmas~\ref{L6.5} and~\ref{L6.8} yield
\[
0<-\Phi_{3,m}(c_3^*)
<
\frac{3}{5m^4}
\]
and
\[
-\Phi_{3,m}'(\xi_{3,m})
>
\frac{20}{9}.
\]
Therefore, \eqref{eq-error-MVT} gives
\[
0<c_3^*-\underline c_{3,m}
<
\frac{3}{5m^4}\frac{9}{20}
=
\frac{27}{100m^4},
\]
which proves \eqref{eq-c3-effective-error}. Thus part~(iii) is proved.

\medskip

\noindent (ii) By Lemma~\ref{L6.7},
\[
\underline c_{d,m}
<
\underline c_{d,m+1}
<
c_d^*,
\qquad
m\in\mathbb N_0.
\]
On the other hand, part~(iii) implies
\[
c_d^*-\underline c_{d,m}\longrightarrow0
\qquad
\text{as }m\to\infty.
\]
Consequently,
\[
\underline c_{d,m}\uparrow c_d^*
\qquad
\text{as }m\to\infty,
\]
which proves \eqref{eq-cdm-convergence} and completes the proof.
\end{proof}

\subsection{Proof of Corollary~\ref{cor:effective-Hd}}

\begin{proof}
By Corollary~\ref{cor:Hd-d23},
\[
\mathcal H_d=\frac{c_d^*}{2},
\qquad
d\in\{2,3\}.
\]
Hence, for \(m\in\mathbb N\), \eqref{eq-c2-effective-error} gives
\[
0<
\mathcal H_2-\frac{\underline c_{2,m}}{2}
=
\frac{c_2^*-\underline c_{2,m}}{2}
<
\frac{1}{15m^3},
\]
which proves \eqref{eq-effective-H2}. Similarly,
\eqref{eq-c3-effective-error} yields
\[
0<
\mathcal H_3-\frac{\underline c_{3,m}}{2}
=
\frac{c_3^*-\underline c_{3,m}}{2}
<
\frac{27}{200m^4},
\]
and therefore \eqref{eq-effective-H3} follows.
\end{proof}

\subsection{Proof of Corollary~\ref{cor:numerical-Hd}}

\begin{proof}
We apply Corollary~\ref{cor:effective-Hd} with \(m=2\).

For \(d=2\) and \(0<c<3\), the recursion
\eqref{eq-shooting-initial-d}--\eqref{eq-shooting-A-d} gives
\[
\mathbb A_2^{(2)}(c)
=
\frac{
50(4c^2-42c+81)
}{
39(c-9)(c-5)
}
\]
and
\[
\mathbb F_3^{(2)}(c)
=
\frac{
20(283c^3-5565c^2+30096c-36855)
}{
63(c-13)(c-9)(c-5)
}.
\]
Substituting these expressions into \eqref{eq-cdm-def} yields
\[
P(c)=0,
\]
where
\[
P(c)
=
405c^4
-
13076c^3
+
141630c^2
-
574812c
+
626535,
\qquad
0<c<3.
\]
We have
\[
P'''(c)=9720c-78456<0,
\qquad
0<c<3.
\]
Hence \(P''\) is strictly decreasing on \((0,3)\), and
\[
P''(c)>P''(3)=91632>0.
\]
Thus \(P'\) is strictly increasing on \((0,3)\), and
\[
P'(c)<P'(3)=-34344<0,
\qquad
0<c<3.
\]
Therefore, \(P\) is strictly decreasing on \((0,3)\). Moreover,
\[
P(1.6890046012)>0,
\qquad
P(1.6890046013)<0.
\]
Hence \(P\) has a unique zero in \((0,3)\), and
\[
1.6890046012
<
\underline c_{2,2}
<
1.6890046013.
\]
Using \eqref{eq-effective-H2} with \(m=2\), we obtain
\[
\frac{1.6890046012}{2}
<
\mathcal H_2
<
\frac{1.6890046013}{2}
+
\frac{1}{120}.
\]
Consequently,
\[
0.8445023006
<
\mathcal H_2
<
0.8528356340,
\]
which proves \eqref{H2-est-m2}.

We now consider \(d=3\). For \(0<c<4\), the recursion
\eqref{eq-shooting-initial-d}--\eqref{eq-shooting-A-d} gives
\[
\mathbb A_2^{(3)}(c)
=
\frac{
2(35c^2-456c+1260)
}{
7(c-10)(c-6)
}
\]
and
\[
\mathbb F_3^{(3)}(c)
=
\frac{
16(1036c^3-25455c^2+187740c-396900)
}{
945(c-14)(c-10)(c-6)
}.
\]
Substituting these expressions into \eqref{eq-cdm-def} yields
\[
Q(c)=0,
\]
where
\[
Q(c)
=
7483c^4
-
297408c^3
+
4130280c^2
-
23224320c
+
42865200,
\qquad
0<c<4.
\]
We have
\[
Q'''(c)
=
179592c-1784448<0,
\qquad
0<c<4.
\]
Hence \(Q''\) is strictly decreasing on \((0,4)\), and
\[
Q''(c)>Q''(4)=2559504>0.
\]
Thus \(Q'\) is strictly increasing on \((0,4)\), and
\[
Q'(c)<Q'(4)=-2542016<0,
\qquad
0<c<4.
\]
Therefore, \(Q\) is strictly decreasing on \((0,4)\). Moreover,
\[
Q(3.6466209429)>0,
\qquad
Q(3.6466209430)<0.
\]
Hence \(Q\) has a unique zero in \((0,4)\), and
\[
3.6466209429
<
\underline c_{3,2}
<
3.6466209430.
\]
Using \eqref{eq-effective-H3} with \(m=2\), we obtain
\[
\frac{3.6466209429}{2}
<
\mathcal H_3
<
\frac{3.6466209430}{2}
+
\frac{27}{3200}.
\]
Consequently,
\[
1.8233104714
<
\mathcal H_3
<
1.8317479715,
\]
which proves \eqref{H3-est-m2}. This completes the proof.
\end{proof}

\vspace{1cm}

\noindent {\bf Data availability.} No data was used for the research described in the article.

\medskip

\noindent {\bf Conflict of interest.} The authors declare that they have no conflict of interest.

\end{document}